\documentclass[a4paper,10pt,oneside,fleqn]{scrartcl}
\usepackage[english]{babel}
\usepackage{amsmath,amsthm,amsfonts,amssymb}
\usepackage{graphicx}
\usepackage[round]{natbib}
\usepackage[usenames,dvipsnames]{color}

\usepackage{subfigure}
\usepackage{multirow}
\usepackage{longtable}
\usepackage{pstricks, pst-node, pst-tree}
\usepackage{hyperref}

\usepackage{array}
\newcolumntype{C}[1]{>{\centering\arraybackslash}p{#1}}

\newcommand{\ec}[1]{$\varepsilon$-constraint#1}
\newcommand{\ecm}{$\varepsilon$-constraint method }

\newcommand{\awt}{augmented weighted Tchebycheff method }

\newcommand{\ve}{\varepsilon}
\newcommand{\wt}{weighted Tchebycheff }

\newcommand{\R}{\mathbb{R}}

\newcommand{\N}{\mathbb{N}}
\newcommand{\NS}{Z_{N}}   

\newcommand{\mbs}{\mathcal{B}_{s}}
\newcommand{\mbss}{\mathcal{B}_{s+1}}
\newcommand{\mbsb}{\overline{\mathcal{B}}_{s}}

\newcommand{\mbsp}{\overline{\mathcal{B}}_{p}}
\newcommand{\tB}{\tilde{B}}
\newcommand{\bB}{\bar{B}}
\newcommand{\hB}{\hat{B}}
\newcommand{\itk}{i=1,\dots,m}
\newcommand{\itkk}{i \in \{1,2,3\}}
\newcommand{\jtk}{j=1,\dots,m}

\newcommand{\Bsi}{B^s_i}
\newcommand{\Bssi}{B^{s+1}_i}
\newcommand{\Bsj}{B^s_j}
\newcommand{\Bssj}{B^{s+1}_j}

\newcommand{\Bssk}{B^{s+1}_k}

\newcommand{\qnr}{quasi non-redundant }
\newcommand{\JB}{\overline{J}_{B}}

\newtheorem{theorem}{Theorem}[section]
\newtheorem{lemma}[theorem]{Lemma}

\newtheorem{corollary}[theorem]{Corollary}
\newtheorem{definition}[theorem]{Definition}
\newtheorem{example}[theorem]{Example}

\newtheorem{assumption}[theorem]{Assumption}

\usepackage{algorithm}
\usepackage{algorithmicx,algpseudocode} 

\newcommand{\algorithmicinput}{\textbf{Input:}}
\newcommand{\INPUT}{\item[\algorithmicinput]}
\newcommand{\algorithmicoutput}{\textbf{Output:}}
\newcommand{\OUTPUT}{\item[\algorithmicoutput]}

\algrenewcommand{\algorithmiccomment}[1]{\hfill // #1} 

\newcommand{\hide}[1]{}
\newcommand{\rh}[1]{\parbox[0pt][#1][c]{0cm}{}}
\newcommand{\addspacerow}{\rh{10pt}}

\allowdisplaybreaks
 
\begin{document}

\title{A linear bound on the number of scalarizations needed to solve discrete tricriteria optimization problems}
\author{Kerstin D\"achert \and Kathrin Klamroth}

\maketitle

\begin{abstract}
Multi-objective optimization problems are often solved by a sequence of parametric single-objective problems, so-called scalarizations. 
If the set of nondominated points is finite, the entire nondominated set can be generated in this way. 
In the bicriteria case it is well known that this can be realized by an adaptive approach which
requires the solution of at most $2|\NS|-1$ subproblems, where $\NS$ denotes the nondominated set of the underlying problem and a subproblem corresponds to a scalarized problem. 
For problems with more than two criteria, no methods were known up to now for which the number of subproblems depends linearly on the number of nondominated points. 
We present a new procedure for finding the entire nondominated set of tricriteria optimization problems for which the number of subproblems to be solved is bounded by $3 |\NS|-2$,
hence, depends linearly on the number of nondominated points. 
The approach includes an iterative update of the search region that, given a (sub-)set of nondominated points, describes the area in which additional nondominated points may be located. 
If the \ecm is chosen as scalarization, the upper bound can be improved to $2 |\NS|-1$. 


\end{abstract}


\section{Introduction} \label{sec:intro}
The determination of the nondominated set is the basis for a multitude of methods in multiple criteria decision making. 
In multiple objective combinatorial optimization the nondominated set is discrete and, assuming that some natural bounds are given for the objective values, also finite.
In this situation, a complete enumeration of all nondominated points can be realized by the successive solution of a series of appropriately formulated scalarized 
problems which are called \emph{subproblems} in what follows. 
Ideally, the total number of 
subproblems depends linearly on the number of nondominated points. 
Indeed, in the bicriteria case approaches are known which require the solution of at most $2|\NS|-1$ subproblems, where $\NS$ denotes the finite  nondominated set of the underlying problem. 
Thereby, $|\NS|$ subproblems are solved to generate all points in $\NS$, and the additional $|\NS|-1$ subproblems are needed to ensure that no further nondominated points exist between the already generated ones \citep[see, e.g.,][]{chalmet86,ralphs06}.
If the \ecm is used, the entire nondominated set of a bicriteria problem can be generated within $|\NS|+1$ subproblems \citep{laumanns06}. 
However, up to now, no linear bounds are known for higher dimensional problems. 
The best known approach has a theoretical bound of $(|\NS|+1)^{m-1}$ subproblems for problems with $m$ objectives \citep{laumanns06}.
In this paper, we present a new procedure for finding the entire nondominated set of tricriteria optimization problems for which the number of scalarized subproblems to be solved is bounded by $3 |\NS|-2$. 
This is achieved by the definition of a new split criterion which allows to exclude redundant parts of the search region. 
It can then be shown that the number of boxes, into which the search region is decomposed, depends linearly on the number of nondominated points. 
If the \ecm is used, the bound can even be improved to $2 |\NS|-1$.

\subsection{Terminology and Definitions}\label{ss:term}

We consider multiple criteria optimization problems 
\begin{equation} \label{optprob}
\min_{x \in X}  \; f(x)=(f_1(x),\dots,f_m(x))^{\top} 
\end{equation}
with $m \geq 2$ objective functions $f_i:X \to \R, \, i=1,\dots,m,$ and with feasible set 
$X \neq \emptyset$. 
Throughout this paper we assume that $X$ is a discrete set. 
The image set of the feasible set $X$ in the outcome space is denoted by $Z:=f(X)$ where $Z$ is a finite set of distinct points in $\R^m$. 

We use the Pareto concept of optimality:  
A solution $\bar{x}\in X$ is called \emph{Pareto optimal} or \emph{efficient} if there does not exist a feasible solution $x\in X$ such that $f_i(x) \leq f_i(\bar{x})$ for all $\itk$ and $f_j(x) < f_j(\bar{x})$ for at least one $j\in \{1,\dots,m\}$. 
The corresponding objective vector $f(x)\in\R^m$ is called \emph{nondominated} in this case. 
If, on the other hand, $f(x) \leq f(\bar{x})$ for some feasible $x\in X$, i.e., $f_i(x) \leq f_i(\bar{x})$ for all $\itk$ and $f_j(x) < f_j(\bar{x})$ for at least one $j\in \{1,\dots,m\}$, we say that $f(x)$ \emph{dominates} $f(\bar{x})$, and $x$ {\it dominates} $\bar{x}$. 
If strict inequality holds for all $m$ components, i.e., if $f_i(x) < f_i(\bar{x})$ for all $\itk$, then $x$ \emph{strictly dominates} $\bar{x}$. 
If there exists no feasible solution $x\in X$ that strictly dominates $\bar{x}$, then $\bar{x}$ is called \emph{weakly Pareto optimal} or \emph{weakly efficient}. 
We denote the set of efficient solutions of~\eqref{optprob} by $X_E$ and refer to it as the \emph{efficient set}. 

The image set of the set of efficient solutions is denoted by $\NS :=f(X_E)$ and is called the \emph{nondominated set} of problem~\eqref{optprob}.
To simplify notation, we will often refer to the points in $Z$ without relating them back to their preimages in the feasible set. Consequently, we equivalently formulate problem~\eqref{optprob} in the outcome space as
\begin{equation}\label{MOoutcome}
\min_{z \in Z} \;  z=(z_1,\dots,z_m)^{\top} \\
\end{equation}
where $Z$ is a discrete set of points in $\R^m$. 
For two vectors $z,\bar{z}\in Z$ we write 
$$\begin{array}{lcl}
z<\bar{z} & \text{if} & z_i<\bar{z}_i \text{~for all~} \itk,\\
z\leq\bar{z} & \text{if} & z_i\leq\bar{z}_i \text{~for all~} \itk \text{~and~} \exists \, j\!\in\!\{1,\dots, m\}:\, z_j<\bar{z}_j, \text{~and}\\
z\leqq\bar{z} & \text{if} & z_i\leq\bar{z}_i \text{~for all~} \itk.
\end{array}$$
The symbols $>$, $\geq$ and $\geqq$ are used accordingly. 

\begin{definition}[Nondominance] \label{def:nondominance}
A point $\bar{z} \in Z$ is called nondominated if and only if there exists no point $z \in Z$ such that $z\leq\bar{z}$.
\end{definition}

A lower bound on the nondominated points of~\eqref{MOoutcome} is given by the \emph{ideal point} which we denote by $z^I$. 
The $i$-th component of the ideal point is defined as the minimum of the $i$-th objective, i.e., $z_i^I := \min \{ z_i \,:\, z \in Z\}$ for all $\itk$. 
A point $z^U$ that strictly dominates $z^I$ is called a \emph{utopia point}. 
Note that in general $z^I \not\in Z$. 
On the other hand, the \emph{nadir point} $z^N$ with components $z_i^N := \max \{ z_i \,:\, z \in \NS\}$ for all $\itk$ provides an upper bound on the nondominated set of~\eqref{MOoutcome}.
While it can be easily determined for bi-objective problems, this is in general not the case for higher dimensional problems. 
However, any upper bound on the nondominated set is sufficient for our purpose. 
Therefore, we typically use upper bounds on $Z=f(X)$ which can be determined also in the presence of more than two criteria. 

A common technique to solve problems of the form~\eqref{optprob} is to iteratively transform the original multiple objective problem into a series of parametric single-objective problems, so called \emph{scalarizations} \citep[see, e.g.,][]{ehrgott05,miettinen99}. 
A variety of different scalarization methods exists which differ, among other things, with respect to their theoretical properties. 
Of particular importance is the question whether the solutions generated by a specific method always correspond to nondominated points of~\eqref{optprob} and whether all nondominated points of~\eqref{optprob} can be generated by appropriately varying the involved parameters.
In this article we do not focus on a specific scalarization, but assume to have 
a scalarization method at hand which possesses these two properties: Every nondominated point can be generated, and the point corresponding to an optimal solution of the scalarization is nondominated. 

\subsection{Literature Review}\label{ss:lit}
The idea of solving parametric single-objective optimization problems in order to generate a (complete) set of nondominated points is well known in the literature.
Especially for bicriteria problems, the literature is rich. 
\citet{aneja79} use a parametric weighted sum method in order to find the extreme supported nondominated points of (integer) linear problems.
They show that the algorithm performs exactly $2n-3$ iterations, where $n$ denotes the number of extreme nondominated points and the two lexicographic minima are assumed to be known. 
\citet{chalmet86} use an \ecm with a weighted sum objective as scalarization for solving bicriteria integer problems. 
Due to the hybrid scalarization, every nondominated point can be computed. 
The authors show that a complete representation is obtained by solving $2|\NS|+1$ integer programs including the computation of the lexicographic optima.
\citet{eswaran89} employ a weighted Tchebycheff scalarization to determine a complete or incomplete representation of the nondominated set of nonlinear integer bicriteria problems. 
\citet{solanki91} use an \awt to generate incomplete representations of mixed integer bicriteria linear programs. 
\citet{ulungu95} address bicriteria combinatorial optimization problems. They introduce a two phase procedure, where all extreme supported nondominated points are computed by weighted sum problems. 
In a second phase, all remaining nondominated points are generated with the help of specific combinatorial methods. 
In \citet{sayin05}, the lexicographic weighted Tchebycheff method and a variant of it serve to solve bicriteria discrete optimization problems. 
Tchebycheff scalarizations are also used by \citet{ralphs06}.
Their algorithm is shown to find all nondominated points by solving $2|\NS|-1$ subproblems including the generation of the lexicographic optima. 
The box algorithm of \citet{hamacher07} uses  lexicographic \ec{} problems.  
While it is designed for incomplete representations, it can also be used to generate the entire nondominated set. 

Also for the discrete multicriteria case, several approaches for finding the entire nondominated set  based on parametric algorithms exist.
\citet{klein82} use a kind of \ecm to determine the entire nondominated set of linear integer problems.
The remaining search region containing possible further nondominated points is described by disjunctive constraints. 
While \citet{chalmet86} mainly address the bicriteria case, they also propose a generalization to the multicriteria case that is based on recursively solving bicriteria problems. 
\citet{tenfelde03} presents a generalization of the two phase method to any number of criteria. 
\citet{sylva04} revisit the idea of \citet{klein82} and reformulate the disjunctive constraints with the help of binary variables. 
\citet{laumanns06} use lexicographic \ec{} problems and show that at most $(|\NS|+1)^{m-1}$ subproblems need to be solved to generate the entire nondominated set.
In the bicriteria case, this yields a total number of only $|\NS|+1$ subproblems, which is smaller than the upper bound of \citet{ralphs06} due to the special scalarization employed.
The numerical experiments for a knapsack problem with three objectives reveal that the number of subproblems needed is considerably less than $(|\NS|+1)^{2}$. 
The authors state that it is an open question whether the number of subproblems can be bounded linearly in terms of the number of nondominated points for problems with more than two criteria.
\citet{laumanns05} improve the algorithm of \citet{laumanns06}. However, no better theoretical bound on the number of subproblems is obtained. 
\cite{oezlen09} 
use an augmented \ecm within a recursive algorithm that is similar to the approach of \citet{chalmet86} and demonstrate that $\mathcal{O}(|\NS|^{m-1})$ iterations are required in the worst case. 
\citet{dhaenens10} extend the approach of \citet{tenfelde03} to a three phase procedure.
Their numerical experiments show that the determination of the nadir point is very expensive regarding computational time.
\citet{przybylski10} also propose a two phase method for integer problems with more than two objectives.
They also encounter the problem of describing the search region and solve it by saving certain upper bound vectors. 
Dominated upper bound vectors are filtered out by pair-wise comparisons.  
\citet{lokman13} build on \citet{sylva04} and propose two improved algorithms based on an augmented $\epsilon$-constraint scalarization. 
While the numerical study of \citet{lokman13} suggests a linear bound on the number of subproblems to be solved in the tricriteria case, only an upper bound of $\mathcal{O}(|\NS|^{2})$ is derived for $m=3$.  
\citet{oezlen13} improve \citet{oezlen09} by saving the right-hand side vectors and the corresponding solutions of the integer problems that have already been solved. Thereby, a huge saving of computational time is achieved.
\citet{kirlik14} improve the method of \citet{laumanns06} by changing the order in which the subproblems are solved. While the numerical results are very competitive, no better theoretical bound on the number of subproblems is proven. 

\subsection{Goals and Outline}

We present an algorithm that generates the entire nondominated set of a discrete tricriteria optimization problem by solving at most $3 |\NS|-2$ subproblems, if $|\NS|\geq 3$ and if bounds on the set of feasible outcomes are given. 
Thereby, to the best of our knowledge, a linear bound with respect to the number of nondominated points is given for the first time for tricriteria problems. 
Our algorithm does not depend on a specific scalarization but can be used with any scalarization method that is suited for discrete and, in general, non-convex problems. 
Our method is also applicable if a subset of nondominated points is already known and the search region potentially containing further nondominated points shall be generated.

The remainder of this paper is organized as follows: 
In Section~\ref{sec:fullsplit} we present a decomposition of the search region based on nondominance and develop a first generic box algorithm. 
We show that this generic algorithm may produce redundant boxes which makes the algorithm inefficient. 
Under the technical assumption that all nondominated points differ pairwise in every component,  we show in Section~\ref{sec:vsplit} how to
construct a decomposition in the tricriteria case that only contains non-redundant boxes. 
The number of boxes is proven to be at most $3 |\NS|-2$ for $|\NS| \geq 3$. 
Finally, we show that the algorithm can also be applied if the nondominated points are in arbitrary position, i.e., every pair of points may have up to $m-2$ equal components. 
The upper bound $3 |\NS|-2$ is also valid in this general case.
In Section~\ref{sec:econstr} we demonstrate that the upper bound can be improved to $2 |\NS|-1$ when an \ecm is used as scalarization.
In Section~\ref{sec:num} we present numerical results.


\section{Split of the search region for multicriteria problems} \label{sec:fullsplit}

Let $B_{0}$ denote an initial search region of the form
$$
B_{0}:=\{z \in \R^m \, : \, l_{j} \leq z_{j} < u_{j}, \, \jtk\}
$$ 
with $l,u \in \R^m, l \leq u$.
As lower and upper bound of $B_{0}$ we choose a global lower and upper bound on the nondominated set, for example, $l:=z^I$ and $u:=z^M$, where $z^I$ is the ideal point and $z^M_{j}:=\max \{z_{j} \, : \, z \in Z \} + \delta$ for all $\jtk$ with $\delta>0$ is an upper bound on the set $Z$. 
If no special scalarization method is employed, the iterative reduction of the search region can solely be based on nondominance. Thereby, every generated nondominated point allows to restrict the search region, as, by Definition~\ref{def:nondominance}, 
for any $z^* \in \NS,$ the two sets
$$
S_{1} (z^*):= \{z \in B_{0} \, : \, z \leqq z^* \}  \quad \textnormal{and} \quad
S_{2} (z^*):= \{z \in B_{0} \, : \, z \geqq z^* \}  
$$
do not contain any nondominated points besides $z^*$, i.e., $S_{1}(z^*)\cap \NS = S_{2}(z^*)\cap \NS = \{z^* \}$. 
Moreover, $S_{1}(z^*) \cap Z=\{z^* \}$, thus, $S_{1}(z^*)\backslash \{z^* \}$ contains no feasible points.

In the following, we decompose a given initial search region 
$B_{0}$ iteratively into subsets $B \subset B_{0}$ of the same form, i.e., into sets 
$B:=\{z \in \R^m \, : \, l_{j} \leq z_{j} < u_{j}', \, \jtk\}$ 
with $u' \in \R^m$, $l \leq u' \leq u$. 
As the initial search region that potentially contains nondominated points of \eqref{optprob} as well as each subset $B$ as defined above describe rectangular subsets of $\R^m$ with sides parallel to the coordinate axes, we call these sets {\emph{boxes} in the following. 
The search region is always represented as the union of certain boxes $B.$ With the generation of every new nondominated point we replace some of the boxes of the current  search region by appropriate new boxes such that the whole search region is covered. 
This property is called correctness in the following. 
\begin{definition}[Correct decomposition] \label{def:correct}
Let $B_{0}$ denote the starting box, let $\mathcal{B}_{s}$ denote the set of boxes at the beginning of iteration $s \geq 1$, where $\mathcal{B}_{1}:=\{B_{0}\}$, and let $z^p \in \NS$, $p=1,\dots,s-1,$ be already determined nondominated points.
We call $\mathcal{B}_{s}$ correct with respect to $z^1,\dots,z^{s-1}$, if 
\begin{equation} \label{eq:correct}
B_{0} \; \backslash \left( \bigcup_{B \in \mathcal{B}_{s}} B \right) =  \bigcup_{p=1,\dots,s-1} {S}_{2}(z^p)
\end{equation}
holds, where ${S}_{2}(z^p):=\{z \in B_{0} : z \geqq z^p \}$
denotes that subset of the box $B_0$ that is dominated by the point $z^p\in \NS, p=1,\dots,s-1$.
\end{definition}
Any split presented in the following maintains a correct decomposition of the search region at any time. Under this basic condition, we try to generate as few boxes as possible, as for every generated box a scalarized subproblem needs to be solved. 
Our aim is to keep the number of subproblems low.
The simplest split decomposes a box $B$ which contains a new outcome $z^* \in (B \cap \NS)$  into $m$ subboxes.

\begin{definition}[Full $m$-split] \label{def:msplit}
Let a nondominated point $z^* \in (B \cap \NS)$ be given. 
We call the replacement of $B$ by the $m$ sets
\begin{equation} \label{eq:fullsplit}
B_{i}:=\{z \in B \, : \, z_{i} < z_{i}^* \}  \; \forall \, \itk 
\end{equation}
a full $m$-split of $B$. 
\end{definition}
Note that similar decomposition approaches are proposed, e.g., in \citet{tenfelde03}, \citet{dhaenens10} and \citet{przybylski10}.
Recursively applying the full $m$-split to every box which contains the current nondominated point yields a correct decomposition, as the following lemma shows.

\begin{lemma}[Correctness of the full $m$-split] \label{lem:correctmsplit}
Let $\mbs, s \geq 1,$ with $\mathcal{B}_{1}:=\{B_{0}\}$ be a correct decomposition with respect to the nondominated points $z^1,\dots,z^{s-1}$, and let $z^s \in \NS$. 
If a full $m$-split is applied to all boxes $B \in \mbs$ with $z^s \in B$, then the resulting decomposition is correct. 
\end{lemma}
\begin{proof} By induction on $s$.\\
$\underline{s=1}:$ Let $\mathcal{B}_{1}:=\{B_{0}\}$, $z^1 \in \NS$. 
Then, by definition of the full $m$-split, $B_{0}$ is replaced by $m$ boxes. It holds that
\[
B_{0} \; \backslash \left( \bigcup_{B \in \mathcal{B}_{2}} B \right) =
B_{0} \; \backslash \left( \bigcup_{\itk} \{z \in B_{0} \, : \, z_{i} < z_{i}^1 \} \right) = {S}_{2}(z^1),
\]
thus, $\mathcal{B}_{2}$ is correct.\\
\noindent
$\underline{s\to s+1}:$ Let $\mathcal{B}_{s}$ be correct and let $z^s \in \NS$. Let $\mbsb \subset \mbs$ denote the set of all boxes $B \in \mbs$ for which $z^s \in B$ holds. 
Let $I$ be the index set of these boxes and let $Q:=|\mbsb|$. 
Now, let a full $m$-split with respect to $z^s$ be applied to all $B \in \mbsb$, i.e., each of the boxes $B^{I(q)}, q=1,\dots,Q$, is replaced by $m$ new boxes $B^{I(q)}_{1},\dots,B^{I(q)}_{m}, q=1,\dots,Q$
and 
$$\bigcup_{\substack{\itk\\ q=1,\dots,Q}} B_{i}^{I(q)}    
=\bigcup_{B \in \mbsb} B \; \backslash \; S_{2}(z^s)
$$
holds. Then
\begin{eqnarray*}
\lefteqn{B_{0} \; \backslash \left( \bigcup_{B \in \mbss} B \right) 
 =B_{0} \; \backslash \left( 
\left( \bigcup_{B \in \mbs \backslash \mbsb} B \right) \cup
\left( \bigcup_{\substack{\itk\\ q=1,\dots,Q}} B_{i}^{I(q)} \right)   
\right)}\\
&\!\!=\!\!&\! B_{0} \; \backslash \left( 
\left( \bigcup_{B \in \mbs \backslash \mbsb} B \right) \cup
\left( \bigcup_{B \in \mbsb} B \; \backslash \; S_{2}(z^s) \right)   
\right)
= B_{0} \; \backslash \left( 
\left( \bigcup_{B \in \mbs} B \; \backslash \; S_{2}(z^s) \right)   
\right)\\
&\!\!=\!\!& \left( B_{0} \; \backslash \left( \bigcup_{B \in \mbs} B  \right) \right) \cup {S}_{2}(z^s)
= \bigcup_{p=1,\dots,s} \!\!\! {S}_{2}(z^p). 
\end{eqnarray*}
\qed 
\end{proof}
Note that all new boxes $B \in \mbss, s \geq 2$, obtained from boxes in $\mbsb$, are defined as sets with open upper boundary, as we need to exclude $z^{s}$ from the search region in order to prevent it from further generation. 
In practical applications, it will often be useful to replace the boxes by closed subsets and exclude $z^s$ by using, for example, appropriate scalarization approaches.

Also note that we describe the boxes by their upper bound $u$ only and that the lower bound of all boxes is kept constant. 
This means that our decomposition contains the union of the sets $S_{1}(z^p) \backslash \{ z^p \}, 1\leq p \leq s$, for all nondominated points $z^p \in \NS$ which have already been generated by the algorithm, even if these sets do not contain any feasible points.
However, the splitting operation is simplified by including these sets, since a box is never split into more than $m$ new boxes in this case.

%
%
\subsection{A generic algorithm based on the full $m$-split}
\begin{algorithm}
\caption{Algorithm with full $m$-split} \label{table:algofullsplit}
\begin{algorithmic}[1]
\INPUT Image of the feasible set $Z \subset \R^m$, implicitly given by some problem formulation
\State $\NS:=\emptyset$; $\delta >0$;
\State \Call{InitStartingBox}{$Z, \delta$};
\State $s:=1;$ \Comment{Initialize starting box}
\While {$\mbs \neq \emptyset$} 
\State Choose $B \in \mbs;$ 
\State $z^{s}:=opt(Z,u(B))$; \Comment{Solve subproblem}
\If {$z^{s} = \emptyset$} \Comment{Subproblem infeasible}
\State $\mbss:=\mbs \backslash \{ B\}$; \Comment{Remove (empty) box}
\Else 
\State $\NS:=\NS \cup \{z^s\}$; \Comment{Save nondominated point}
\State $\mbss:=\mbs;$ \Comment{Copy set of current boxes}
\State \Call{GenerateNewBoxes}{$\mbs,z^s,z^I,\mbss$};
\EndIf
\State $s:=s+1$; 
\EndWhile
\OUTPUT Set of nondominated points $\NS$
\Statex
\Procedure{InitStartingBox}{$Z, \delta$}
\For {$j=1$ to $m$} \Comment{Compute bounds on $Z$}
\State $z^I_{j}:=\min \{z_{j} \, : \, z \in Z \}$;  
\State $z^M_{j}:=\max \{z_{j} \, : \, z \in Z \} + \delta$; 
\State $u_{j}(B_{0}):=z^M_{j}$;
\EndFor
\State $\mathcal{B}_{1}:=\{B_{0}\};$ \Comment{Initialize set of boxes}
\State \Return $\mathcal{B}_{1}$
\EndProcedure
\Statex
\Procedure{GenerateNewBoxes}{$\mbs,z^s,z^I,\mbss$}
\ForAll {$\hB \in \mbs$}
\If {$z^s<u(\hat{B})$} \Comment{Point is contained in box}
\For {$i=1$ to $m$} \Comment{Apply full $m$-split to $\hB$}
\If {$z_{i}^s>z_{i}^I$}
\State $u(\hB_{i}):=u(\hB)$; \Comment{Create a copy of $\hB$} 
\State $u_{i}(\hB_{i}):=z_{i}^s;$ \Comment{Update upper bound} 
\State $\mbss:=\mbss \cup \{\hB_{i} \};$ \Comment{Append new box} 
\EndIf
\EndFor
\State $\mbss:=\mbss \backslash \{\hB\};$ \Comment{Remove box}
\EndIf 
\EndFor
\State \Return $\mbss$
\EndProcedure
\end{algorithmic}
\end{algorithm} 
Algorithm~\ref{table:algofullsplit} shows a basic algorithm using the full $m$-split.
Due to Lemma~\ref{lem:correctmsplit}, the algorithm is correct as it does not exclude regions from the search region which might contain further nondominated points. 
A problem formulation is given as input, which is denoted by $Z$. 
Note that this does not mean that the set of feasible outcomes is known explicitly, but it is to be understood as a substitute for the objective functions and the constraints. 
As long as $\mbs$ contains unexplored boxes, a box $B$ is selected according to some specified rule.
As we are interested in generating the entire nondominated set, no special rule is employed in the following, and we may, for example, always take the first box in the list $\mbs$. 
The upper bounds of the chosen box $B$ are used to determine the parameters of the selected scalarization. 
Thereby, the scalarization method can be chosen freely, as long as it is guaranteed that the method finds a nondominated point in $B$ whenever there exists one. 
For example, the augmented weighted Tchebycheff scalarization is an appropriate method, see, e.g., \cite{daechert12} and \cite{daechert14} for an adaptive parameter choice in the bicriteria and multicriteria case, respectively. 
The result of the subproblem is either a nondominated point~$z^{s}$ in the considered box~$B$ or the detection of infeasibility,
which corresponds to the situation that $B$ does not contain further nondominated points. 
In the latter case, $B$ is removed from the list $\mbs$ and the iteration is finished. Otherwise, $z^s$ is saved and all boxes $\hB \in \mbs$ are identified that contain $z^s$. 
All these boxes are split with respect to all $i \in \{1,\dots,m\}$ for which $z^s_{i}>z_{i}^I$ holds and are replaced by the new boxes. 
The algorithm iterates until all boxes have been explored. Then the entire nondominated set has been detected.

%
\subsection{Bicriteria case}

For $m=2$, Algorithm~\ref{table:algofullsplit} is not only correct but also efficient, in the sense that the number of subproblems that need to be solved depends linearly on the number of nondominated points. 
As the decomposition does not contain redundant boxes, an upper bound on the number of boxes can easily be derived, which can be seen as follows.  
Let $B_0$ denote the starting box and let $z^1\in B_0\cap\NS$ be the first generated point. 
Consider the two new boxes $B_{1},B_{2}$ replacing $B_{0}$ in the first iteration. It holds that
\begin{align*}
B_{1} \cap Z & =\{z \in B_{0} \, : \, z_{1} < z_{1}^1 \} \cap Z =
\left( \{z \in B_{0} \, : \, z_{1} < z_{1}^1 \} \cap Z \right) \backslash \, S_{1}(z^1) \\ 
& =\{z \in B_{0} \, : \, z_{1} < z_{1}^1, z_{2}>z_{2}^1 \} \cap Z
\intertext{and, analogously,}
B_{2} \cap Z &=\{z \in B_{0} \, : \, z_{2} < z_{2}^1 \} \cap Z =
\{z \in B_{0} \, : \, z_{2} < z_{2}^1, z_{1}>z_{1}^1 \} \cap Z,
\end{align*}
thus, $(B_{1} \cap Z) \cap (B_{2} \cap Z) = \emptyset$. 
Therefore, the second generated point $z^2 \in \NS$ is contained in exactly one of the two boxes $B_{1},B_{2}$. 
This box is again split into two new boxes whose intersections with $Z$ are disjoint among themselves as well as from the box (intersected with $Z$) which has not been changed in the current iteration. 
Repeating this argument, we see that for $m=2$, no redundancy occurs. 
Therefore, we can easily indicate the running time of Algorithm~\ref{table:algofullsplit} in the bicriteria case based on the knowledge that a new nondominated point lies in exactly one box. 
In the initialization phase, $z^I$ and $z^M$ are computed in order to define~$B_{0}$. 
In every iteration, either a (new) nondominated point is generated or a box is discarded from the search region. 
For every new nondominated point $z^s>z^I$, two new boxes replace the currently investigated box, and for each of the two lexicographic optimal points (defining the ideal point) the current box is replaced by one new box. 
So, the total number of iterations is $2 |\NS|-1$, cf.\ \citet{chalmet86} and \citet{ralphs06}.  
In Figure~\ref{fig:2d}, we illustrate the search region after having generated and inserted four nondominated points $z^{1}, z^{2}, z^{3}$ and $z^4$. 
If we assume that these points build the entire nondominated set, Algorithm~\ref{table:algofullsplit} terminates after seven iterations.

\begin{figure}
\centering
\includegraphics[width=.4\textwidth]{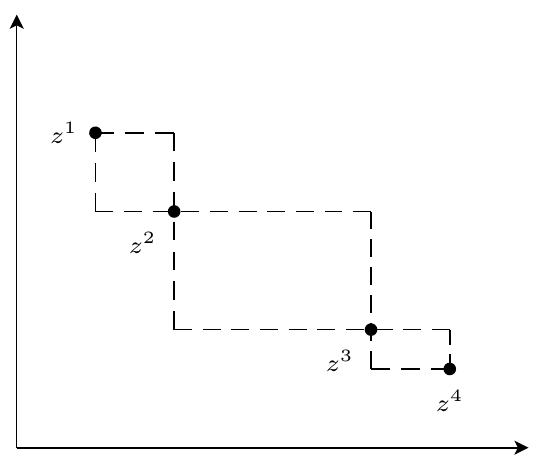} 
\caption{Decomposition of the search region for $m=2$}
\label{fig:2d}
\end{figure}

\subsection{Multicriteria case ($m \geq 3$)}

The application of 
the full $m$-split in the case of three criteria is illustrated in Figure~\ref{fig:genericsplit}. 
Different from the bicriteria case, 
a nondominated point may lie in the intersection of multiple boxes for $m \geq 3$. 
If we perform the full $m$-split in every box which contains the current nondominated point, we typically create nested and, thus, redundant subboxes. 
This is illustrated in the following example.
\begin{example} \label{ex:split}
Let $m=3$ and let the initial search region be given by
$$
B_0:=\{z \in Z \, : \, 0 \leq z_{i} \leq 5 \; \forall \, i=1,2,3 \}.
$$
Assume that the first nondominated point that is generated is $z^{1}=(2,2,2)^{\top}$. 
Performing a full $3$-split in $B_0$ with respect to $z^{1}$ replaces the search region $B_{0}$ by the three sets
$$
B_{1,i}:=\{z \in B_{0} \, : \, z_{i} < 2 \}, \; i=1,2,3. 
$$
Let $z^{2}=(1,1,4)^{\top}$ be the next nondominated point that is generated. 
It holds that $z^{2} \in B_{11}$ as well as $z^{2} \in B_{12}$, but $z^{2} \notin B_{13}$. 
Performing a full $3$-split in $B_{11}$ yields
\begin{align*}
B_{21} & :=\{z \in B_{0} \, : \, z_{1} < 1 \},\\
B_{22} & :=\{z \in B_{0} \, : \, z_{1} < 2, z_{2}<1 \},\\
B_{23} & :=\{z \in B_{0} \, : \, z_{1} < 2, z_{3}<4 \}.
\end{align*}
Performing a full $3$-split in $B_{12}$ yields
\begin{align*}
B_{21}' & :=\{z \in B_{0} \, : \, z_{1} < 1, z_{2}<2 \},\\
B_{22}' & :=\{z \in B_{0} \, : \, z_{2}<1 \},\\
B_{23}' & :=\{z \in B_{0} \, : \, z_{2} < 2, z_{3}<4 \}.
\end{align*}
It holds that $B_{21}' \subset B_{21}$ and $B_{22} \subset B_{22}'$, thus, the boxes $B_{22}$ and $B_{21}'$ are redundant in the decomposition of $B_{0}$. 
\end{example}

If redundant boxes are kept in the decomposition, this typically increases the running time of the algorithm, as additional, unnecessary subproblems are solved. 
Depending on the given problem, this may be time-consuming. 
Thus, redundant boxes should be detected and removed immediately. 
In the following, we analyze under which conditions redundant boxes occur. 
We first define our notion of non-redundancy.

\begin{figure}
\centering
\includegraphics[width=.9\textwidth]{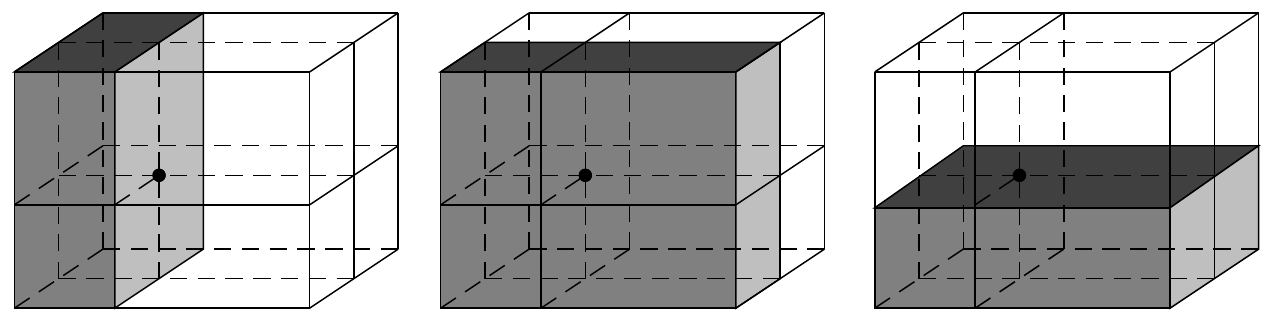}
\caption{Boxes $B_{i},i=1,2,3$, in $\R^3$ obtained by a full $3$-split of the initial search region; the nondominated point with respect to which the split is performed is represented by a dot}
\label{fig:genericsplit}
\end{figure}

\begin{definition}[Non-redundant decomposition] \label{def:nonred1}
Let $B_{0}$ denote the starting box and let $\mathcal{B}_{s}$ be a correct decomposition at the beginning of iteration $s \geq 1$. 
We call $\mbs$ (and every $B \in \mbs$) non-redundant, if for every pair of boxes $B, \tilde{B} \in \mbs, B \neq \tilde{B}$, it holds: 
$$\exists \, i \in \{1,\dots,m\} : u_{i}(B) < u_{i}(\tB) \; \text{and} \; \exists \, j \in \{1,\dots,m\} :
u_{j}(B) > u_{j}(\tB).$$
In case that $u(B) \leq u(\tB)$ we say that box $\tB$ dominates $B$. 
\end{definition}

Note that the definition of a dominated box is somehow opposite to the definition of a dominated point. While $u \in \R^m$ is dominated by 
$u' \in \R^m$ if $u \geq u'$, box $B$ is dominated by $B'$ if $u(B) \leq u(B')$. 

For simplicity, we make a technical assumption concerning the values of the nondo\-minated points that  will be removed later. 
Moreover, we define our general setting.
\begin{assumption} \label{assump:znew}
Let the following hold:
\begin{enumerate}
\item For all nondominated points $z^p \in \NS, p=1,\dots,s$, generated up to iteration $s \geq 1$, it holds that $z_{j}^p \neq z_{j}^q$ for all $\jtk$ and $1 \leq q < p$.
\item The starting box $B_{0}$ is non-empty, and $\mathcal{B}_{1}:=\{ B_{0}\}$ denotes the initial decomposition of the search region. 
\item 
For every $1 \leq p \leq s$, $\mathcal{B}_{p}$ is a correct, non-redundant decomposition of the search region. 
By $\mbsp := \{ B \in \mathcal{B}_{p} : z^p \in B\}$ we denote the subset of boxes in iteration $p$ containing $z^p$.
\end{enumerate}
\end{assumption}

\begin{lemma}[Generation of redundant boxes] \label{lem:genredbox}
Let Assumption~\ref{assump:znew} be satisfied.
If we apply a full $m$-split to every box $B \in \mbsb$, then redundancy can only occur among the `descendants' of two different boxes which have been split with respect to the same component in this iteration. 
\end{lemma}
\begin{proof}
We first show that no redundancy occurs between two boxes if at least one of the boxes has not been changed in the current iteration. Therefore, consider two arbitrary boxes $B, \tB \in \mbs$ where $\tB \in \mbs \backslash \mbsb$:
\begin{enumerate}
\item If $B \in \mbs \backslash \mbsb$, both boxes remain unchanged in the current iteration and, thus, due to Assumption~\ref{assump:znew}~(3) are non-redundant. 
\item If $B \in \mbsb$, none of the boxes obtained from a split in $B$ can dominate $\tB$, as $B$ does not dominate $\tB$ and the upper bound of $B$ is only decreased by the split.
Conversely, $\tB$ cannot dominate any of the boxes obtained from a split in $B$, 
since $\tB \in \mbs \backslash \mbsb$ implies that $u_{j}(\tB) \leq z^s_{j}$ for at least one $j\in\{1,\dots,m\}$. 
As $z^s<u(B)$, for every $B_{i}, \itk$, resulting from a split of~$B$ it holds that $z^s_{i} = u_{i}(B_{i})$ and $z^s_{j} < u_{j}(B_{i})$ for all $j \neq i$. Thus, $\tB$ dominates $B_{i}$ if and only if $z^s_{i}=u_{i}(\tB)$ holds. This, however, is excluded by Assumption~\ref{assump:znew}~(1).
\end{enumerate}
Therefore, redundancy can only occur among newly generated boxes. Consider two boxes $B_{i} \neq \hB_{j}$ obtained from $B,\hB \in \mbsb$ (the case $B=\hB$ is included) that are  split with respect to components $i \neq j$.  
Then $u_{i}(B_{i})=z^s_{i}<u_{i}(\hB_{j})$ and $u_{j}(B_{i}) > z^s_{j}=u_{j}(\hB_{j})$, thus, none of the boxes can dominate the other one. 
It follows that redundancy can only occur among the descendants of two different boxes split with respect to the same component.
\qed 
\end{proof}
\begin{corollary} \label{cor:oneboxsplit}
Let Assumption~\ref{assump:znew} hold.
If only one box is split in some iteration, then all $m$ resulting subboxes are non-redundant.
In particular, the boxes obtained in the first iteration are always non-redundant.
\end{corollary}

\begin{corollary} \label{cor:nonred}
Let Assumption~\ref{assump:znew} hold. Let two boxes $B, \hB \in \mbsb$ be split with respect to the same component $\itk$. Then the resulting boxes $B_{i},\hB_{i}$ are non-redundant if and only if there exists an index 
$p \neq i$ such that $u_{p}(B_{i}) < u_{p}(\hB_{i})$ and there exists an index $q \neq i$ such that $u_{q}(B_{i}) > u_{q}(\hB_{i})$. 
\end{corollary}

These observations allow to detect redundant boxes by checking specific boxes of the current decomposition. 
Translating this to an algorithm for arbitrary $m \geq 3$, we apply, in every iteration, a full $m$-split to every box containing the current nondominated point. 
If more than one box is split in one iteration, we compare the upper bounds of those new boxes that were generated with respect to the same component pairwise to detect redundancy. The respective boxes are then removed from the decomposition. 
A similar approach is followed in \citet{przybylski10} where a new upper bound is compared to all other upper bounds excluding the one from which it has been derived. In doing so, `dominated' upper bounds are filtered out directly. 

A corresponding algorithm can be improved further if 
redundant boxes are identified without comparing their upper bounds to some or all other boxes. 
In the next section we develop an explicit criterion for tricriteria problems that indicates already before the split is performed whether the resulting box is redundant or not, and, thus, allows to maintain only non-redundant boxes in the decomposition. 
We prove that the number of non-redundant boxes or, equi\-valently, the number of subproblems to be solved in the course of an algorithm based on such an improved split operation depends linearly on the number of nondominated points. 


\section{A split criterion to avoid redundant boxes for $m = 3$} \label{sec:vsplit}

According to Definition \ref{def:nonred1}, a non-redundant 
box can be characterized as follows. 
A box is non-redundant if and only if it contains a non-empty subset which is not part of any other box of the decomposition. These subsets are studied in the following.
\begin{definition}[Individual subsets] \label{def:nonred}
Let $\mathcal{B}_{s}, s \geq 1$ be a non-redundant decomposition.  
For every $B \in \mbs$, the set 
\begin{equation} \label{eq:defV}
V(B):=B \; \backslash \left( \bigcup_{\tB \in \mathcal{B}_{s} \backslash \{B\}} \tB \right)
\end{equation}
is called individual subset of $B.$
\end{definition}

\begin{figure}
\centering
\includegraphics[width=.62\textwidth]{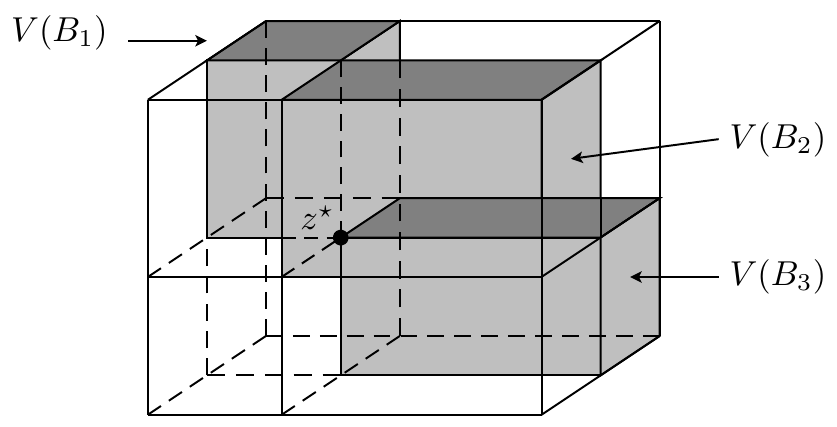} 
\caption{Individual subsets $V(B_{i})$, $i=1,2,3$, in $\R^3$ obtained by a full $3$-split of the initial search region with respect to $z^{\star} \in \NS$}
\label{fig:V}
\end{figure}

Obviously, for every $B \in \mbs, s \geq 1$, it holds that $V(B) \subseteq B$ and $V(B) \cap V(\tB)=\emptyset$ for every $\tB \in \mbs, \tB \neq B$. 
Figure~\ref{fig:V} shows the individual subsets of the three boxes $B_{i}, i=1,2,3$, in $\R^3$, obtained by a full $3$-split of the initial box, which are depicted in Figure~\ref{fig:genericsplit}.

Now, maintaining only non-redundant boxes in the decomposition of the search region is equivalent to maintaining boxes with non-empty individual subsets. 
An explicit split criterion should indicate already before performing the split whether a given box will have a non-empty individual subset after having performed the split. 
To this end, we have to describe the individual subsets explicitly.  
For $m=3$, we observe that the individual subset of a box is bounded by the neighbors of that box. 
After defining the neighbor of a box with respect to a certain component, we show its existence and indicate the respective neighboring boxes by a constructive proof.

\begin{definition}[Neighbor of a box] \label{def:neighbors}
Let $\mathcal{B}_{s}, s \geq 1$ be a non-redundant decomposition of the search region, and let 
$\underline{u}_{i}:=\min\{u_{i}(B) : B \in \mbs \}$.
Let any $\bB \in \mbs$ be given. For every $i \in \{1,2,3\}$, for which $u_{i}(\bB)> \underline{u}_{i}$, we call a box $\hB \in \mbs \backslash \{ \bB \}$ that satisfies 
 \begin{align}  
 u_{i}(\hB) & < u_{i}(\bB) \label{def:neigh1} \\
 u_{j}(\hB) & > u_{j}(\bB) \quad \text{for some} \; j \neq i \label{def:neigh2}\\
 u_{k}(\hB) & \geq u_{k}(\bB) \quad \text{for} \; k \neq i,j \label{def:neigh4}
\intertext{and}
u_{i}(\hB)&=\max\{u_{i}(B) : B \in \mbs \backslash \{ \bB \}, u_{i}(B) < u_{i}(\bB) \} \label{def:neigh3}
\end{align}
the neighbor of $\bB$ with respect to $i$ at the beginning of iteration $s$, denoted by $\Bsi(\bB)$.  
\end{definition}

\begin{example}
Consider Figure~\ref{fig:genericsplit}, which depicts the three boxes that are obtained in the first iteration. At the beginning of iteration $s=2$,  
it holds that $B^2_{1}(B_{2})=B_{1}$, since $B_{1}$ is the unique box satisfying \eqref{def:neigh1}--\eqref{def:neigh3} for $\bB:=B_{2}$. 
Analogously, $B^2_{3}(B_{2})=B_{3}$ holds. 
A neighbor $B^2_{2}(B_{2})$ is not defined as $u_{2}(B_{2})=\underline{u}_{2}$.
\end{example}

The following lemma shows that, under appropriate assumptions, for every box $\bB\in\mbs$ and every component $\itkk$ for which $u_{i}(\bB)> \underline{u}_{i}$ holds there exists a unique neighbor $\Bsi(\bB)$ satisfying \eqref{def:neigh1}--\eqref{def:neigh3} of Definition~\ref{def:neighbors}.
These neighbors $\Bsi(\bB), \itkk$, which will be indicated with the help of a constructive proof will turn out to be the boxes that define the individual subset of $\bB$. 

\begin{assumption} \label{assump:znew2}
Let the following hold:
\begin{enumerate}
\item For all nondominated points $z^p \in \NS, p=1,\dots,s$, generated up to iteration $s \geq 1$, it holds that $z_{j}^p \neq z_{j}^q$ for all $j \in \{1,2,3\}$ and $1 \leq q < p$.
\item The starting box $B_{0}$ is non-empty, and $\mathcal{B}_{1}:=\{ B_{0}\}$ denotes the initial decomposition of the search region. 
\item 
For every iteration $1 \leq p \leq s$, the set $\mathcal{B}_{p+1}$ is obtained from $\mathcal{B}_{p}$ by applying a full $3$-split to every $B \in \mbsp$, where $\mbsp := \{ B \in \mathcal{B}_{p} : z^p \in B\}$.
 All redundant boxes are removed from $\mathcal{B}_{p+1}$ at the end of the respective iteration $p$. 
 \end{enumerate}
\end{assumption}

Note that Assumption~\ref{assump:znew2} substantiates Assumption~\ref{assump:znew} by specifying that the correct, non-redundant decompositions are obtained by iterative full $3$-splits and that redundant boxes are removed.

As the proof of the following lemma is rather technical, we illustrate it with the help of two examples. 
In both examples, $u(B_{0}):=(5,5,5)^{\top}$ is assumed, and $z^1:=(2,2,2)^{\top}$ is inserted as a first nondominated point. 
In the first example, depicted in Figure~\ref{fig:visualizelemma1}, 
$z^2:=(3,1,4)^{\top}$ is inserted as a second nondominated point.
In the second example, depicted in Figure~\ref{fig:visualizelemma2},
$z^2:=(1,1,4)^{\top}$ represents the second nondominated point. 

\begin{lemma} \label{lem:neighb}
Let Assumption~\ref{assump:znew2} be satisfied. 
Then, for every $s \geq 2$, every $\bB \in \mbs$ and every $i \in \{1,2,3\}$, for which $u_{i}(\bB)> \underline{u}_{i}:=\min\{u_{i}(B) : B \in \mbs \}$ holds, there exists a unique neighbor $\Bsi(\bB) \in \mbs$ satisfying \eqref{def:neigh1}--\eqref{def:neigh3}. Particularly, $u_{k}(\Bsi(\bB)) = u_{k}(\bB)$ holds, i.e., $\Bsi(\bB)$ satisfies 
 \begin{align} 
 u_{i}(\Bsi(\bB)) &< u_{i}(\bB) \label{neighi} \\
 u_{j}(\Bsi(\bB)) &> u_{j}(\bB) \quad \text{for some} \; j \neq i \label{neighj}\\
 u_{k}(\Bsi(\bB)) &= u_{k}(\bB) \quad \text{for} \; k \neq i,j, \label{neighk}
\intertext{and}
u_{i}(\Bsi(\bB))&=\max\{u_{i}(B) : B \in \mbs, u_{i}(B) < u_{i}(\bB) \}. \label{neighimax}
 \end{align} 
If $u_{i}(\bB)=\underline{u}_{i}$, then we set $\Bsi(\bB):=\emptyset$. 
\end{lemma}
\begin{proof}
By induction on $s$.\\
$\underline{s=2}: \mathcal{B}_{1}=\{ B_{0}\}=\overline{\mathcal{B}}_{1}$, as $z^1<u(B_{0})=z^M$. The starting box is split into three subboxes $\hB_{i}:=\{z \in B_{0} : z_{i} < z_{i}^1 \}, i \in \{1,2,3\}$. Due to Lemma~\ref{lem:genredbox}, the new boxes are non-redundant, thus, $ \mathcal{B}_{2}=\{\hB_{1},\hB_{2},\hB_{3} \}$. Consider $\hB_{i}$ for fixed $i \in \{1,2,3\}$: As $u_{i}(\hB_{i})=z_{i}^1 < u_{i}(\hB_{j})$ for every $j \neq i$, by definition, $B_{i}^2(\hB_{i})=\emptyset$. As $u_{j}(\hB_{i})>z_{j}^1=\min\{u_{j}(B) : B \in \mathcal{B}_{2} \}$, for every $j \neq i,$ a unique neighbor $B_{j}^2(\hB_{i})$ exists and, obviously, $B_{j}^2(\hB_{i})=\hB_{j}$ holds.
\\
$\underline{s \to s+1}$: We assume that unique neighbors satisfying \eqref{def:neigh1}--\eqref{def:neigh3} exist for all $B \in \mbs$ and that, additionally, \eqref{neighk} holds. We insert $z^s \in \NS$. Due to the correctness of the full $m$-split, there exists at least one box which is split, i.e., $|\mbsb| \geq 1$.

Case 1: $|\mbsb| = 1$, i.e., only one box is split. 
Let $\hB$ be this box and let $\hB_{i}, i \in \{1,2,3\}$, be the subboxes resulting from the split. 
Due to Lemma~\ref{lem:genredbox}, the new boxes are non-redundant, thus, $ \mbss=(\mbs \backslash \{\hB\}) \cup \{\hB_{1},\hB_{2},\hB_{3} \}$. 
The corresponding box in Figure~\ref{fig:visualizelemma1}~(a) is $\hB=B_{12}$, which is replaced by $\hB_{1}=B_{21}$, $\hB_{2}=B_{22}$, $\hB_{3}=B_{23}$, see Figure~\ref{fig:visualizelemma1}~(b). 

\begin{figure}
\centering
\subfigure[Before the insertion of $z^2$]{
\includegraphics[width=.47\textwidth]{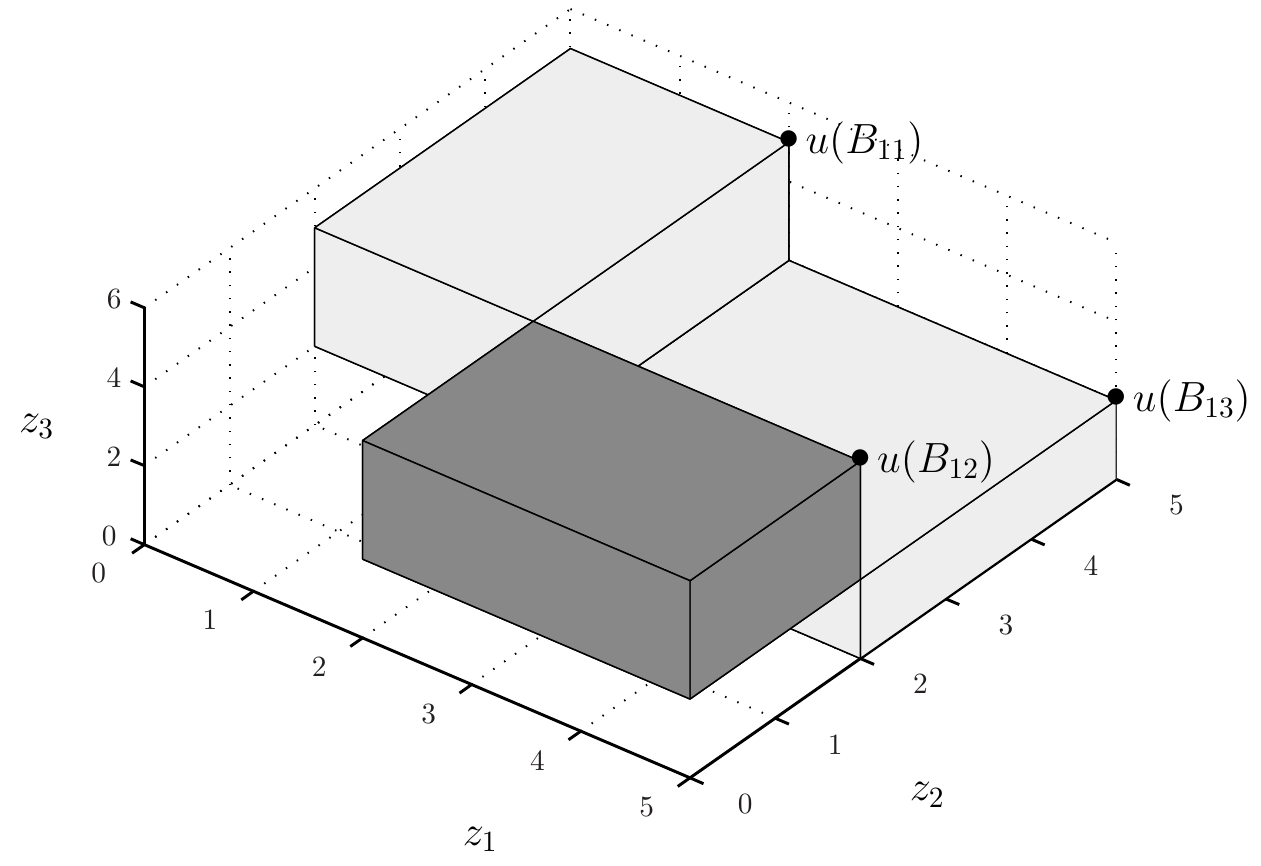} 
}
\subfigure[After the insertion of $z^2$]{
\includegraphics[width=.47\textwidth]{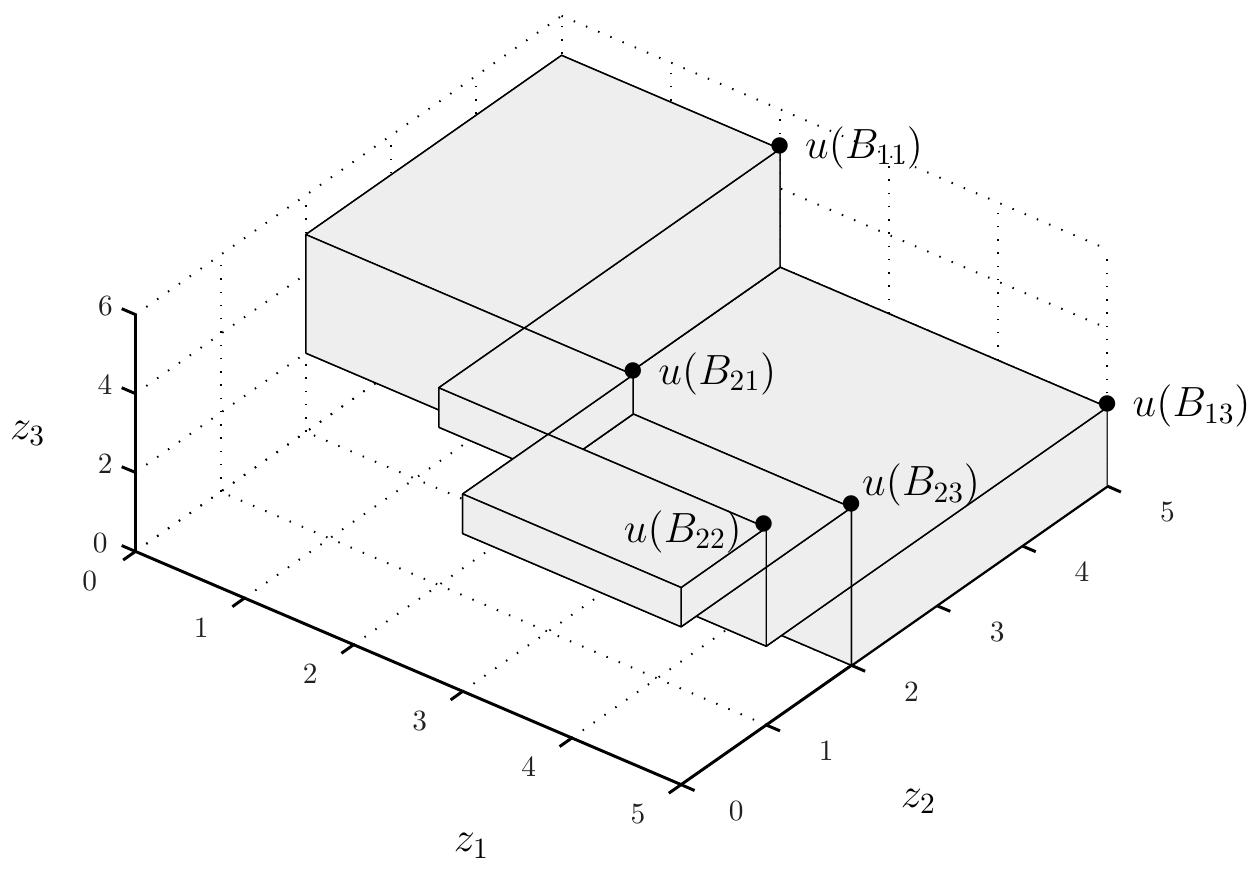} 
} 
\caption{Visualization of the upper bound vectors $u(B)$ in case~1 of the proof of Lemma~\ref{lem:neighb}: 
$z^2=(3,1,4)^{\top}$ lies only in box  
$B_{12}$ with $u(B_{12})=(5,2,5)^{\top}$, i.e.,  $|\overline{\mathcal{B}}_{2}| = 1$. 
For a better illustration, the individual subsets $V(B)$ of all boxes are depicted. 
} 
\label{fig:visualizelemma1} 
\end{figure}

Consider an arbitrary box $\hB_{i}, i \in \{1,2,3\}$. Then the following holds:
\begin{itemize}
\item[(i)] $\Bssi(\hB_{i})=\Bsi(\hB)$:\\ 
$\hB_{i}$ is the only new box $B$ with $u_{i}(B) = z^s_{i}$ and there is no other new box $B$ satisfying $u_{i}(B) < z^s_{i}$. Hence, $\Bssi(\hB_{i}) \notin \{\hB_{1},\hB_{2},\hB_{3} \}$.  If $\Bsi(\hB)=\emptyset$ then $\Bssi(\hB_{i})=\emptyset$. Otherwise, i.e., if $\Bsi(\hB)$ exists, 
$u_{i}(\Bsi(\hB))<u_{i}(\hB)$, $u_{j}(\Bsi(\hB)) > u_{j}(\hB)$ for some $j \neq i$ and $u_{k}(\Bsi(\hB)) = u_{k}(\hB)$ for $k \neq i,j$ hold due to the induction hypothesis. Now $u_{i}(\Bsi(\hB)) \leq z^s_{i}$ must be satisfied, as otherwise $\Bsi(\hB_{i}) \in \mbsb$ would hold, a contradiction to the assumption that $\mbsb=\{\hB\}$. Moreover, $u_{i}(\Bsi(\hB)) = z^s_{i}$ is excluded due to Assumption~\ref{assump:znew2}~(1).
Thus, \eqref{neighi}-\eqref{neighimax} holds for $\Bssi(\hB_{i})=\Bsi(\hB)$. The uniqueness of $\Bsi(\hB)$ follows from the induction hypothesis.

In Figure~\ref{fig:visualizelemma1}~(b), an example of this case is given by $B_{1}^3(B_{21})=B^2_{1}(B_{12})=B_{11}$. 

\item[(ii)] $\Bssj(\hB_{i})=\hB_{j}$ for all $j \neq i$:\\
$\hB_{j}$ is the only new box $B$ with $u_{j}(B) = z^s_{j}$ and there is no other new box $B$ satisfying $u_{j}(B) < z^s_{j}$. Furthermore, $\hB_{j}$ satisfies \eqref{neighi}-\eqref{neighk}, as 
$u_{j}(\hB_{j}) < u_{j}(\hB_{i})$,
$u_{i}(\hB_{j}) > u_{i}(\hB_{i})$ and $u_{k}(\hB_{j}) = u_{k}(\hB_{i})$ for $k \neq i,j$ hold. Moreover, $u_{j}(\hB_{j})$ is maximal, as $u_{j}(\hB_{j})=z^s_{j} \geq u_{j}(\Bsj(\hB))$ if $\Bsj(\hB) \neq \emptyset$ and $u_{j}(\Bsj(\hB))$ maximal due to the induction hypothesis. As $z^s_{j} = u_{j}(\Bsj(\hB))$ is excluded due to Assumption~\ref{assump:znew2}~(1), the uniqueness of $\Bssj(\hB_{i})$ follows.

In Figure~\ref{fig:visualizelemma1}~(b), examples are given by $B_{2}^3(B_{21})=B_{22}$ and $B^3_{3}(B_{21})=B_{23}$. 

\end{itemize}

Now consider an arbitrary box $B \neq \hB$. Then the following holds:
\begin{itemize}
\item[(iii)] If $\Bsi(B) \neq \hB$ for some $i \in \{1,2,3\}$, then $\Bssi(B)$ remains unchanged:\\
Assume that $\Bssi(B)$ changes due to the split of $\hB$. Then the only candidate for $\Bssi(B)$ is $\hB_{i}$ and only in case that $u_{i}(\hB_{i})<u_{i}(B) \leq u_{i}(\hB)$, as otherwise $\Bsi(B)=\hB$ would have been valid. Now suppose that $\Bssi(B)=\hB_{i}$. As $u_{j}(\hB_{i})=u_{j}(\hB)$ for all $j \neq i$ and, by definition of $\Bsi(B)$, $u_{j}(\hB_{i}) \geq u_{j}(B)$ for all $j \neq i$, we have that $u_{j}(\hB) \geq u_{j}(B)$ for all $j \neq i$ and hence $u_{l}(\hB) \geq u_{l}(B)$ for all $l \in \{1,2,3\}$, a contradiction to $\mbs$ being non-redundant.
Thus, $\Bssi(B)$ remains unchanged.

In the example depicted in Figure~\ref{fig:visualizelemma1}~(b), let $B=B_{13}$. 
As $B_{1}^2(B_{13})=B_{11} \neq B_{12}$, the neighbor remains unchanged, thus, $B^3_{1}(B_{13})=B_{11}$. 

\item[(iv)] If $\Bsi(B) = \hB$ for some $\itk$, then $\Bssi(B)=\hB_{j}$ with $j$ being the unique index for which $u_{j}(\hB)>u_{j}(B)$ holds:\\
By the induction hypothesis, 
$u_{i}(\hB)<u_{i}(B)$, $u_{j}(\hB)>u_{j}(B)$ for some $j \neq i$ and $u_{k}(\hB)=u_{k}(B)$ for $k \neq i,j$. As $z^s_{i}=u_{i}(\hB_{i})<u_{i}(\hB)$ and $u_{l}(\hB_{i})=u_{l}(\hB)$ for all $l \neq i$,  $\hB_{i}$ is a candidate for $\Bssi(B)$. As $B \notin \mbsb$ and $z_{l}^s < u_{l}(B)$ for all $l \neq j$ it follows that $z^s_{j} \geq u_{j}(B)$, and, due to Assumption~\ref{assump:znew2}~(1), $z^s_{j} > u_{j}(B)$. Thus,  
$u_{i}(\hB_{j})=u_{i}(\hB)<u_{i}(B)$, $u_{j}(\hB_{j})=z^s_{j}>u_{j}(B)$ and $u_{k}(\hB_{j})=u_{k}(\hB)=u_{k}(B)$ hold. Therefore, $\hB_{j}$ is the unique other candidate for $\Bssi(B)$ besides $\hB_{i}$. As $u_{i}(\hB_{i})<u_{i}(\hB_{j})=u_{i}(\hB),$ $\hB_{j}$ is the unique neighbor $\Bssi(B)$ after the split.

In the example depicted in Figure~\ref{fig:visualizelemma1}~(b), 
$B_{2}^2(B_{13})=B_{12} = \hB$ and $B^3_{2}(B_{13})=B_{23}$. Box $B_{22}$ is the unique other candidate for $B^3_{2}(B_{13})$, but as 
$u_{2}(B_{22})<u_{2}(B_{23})$ it holds that  $B^3_{2}(B_{13})=B_{23}$. 
\end{itemize}
Case 2: $|\mbsb| > 1$.\\
By definition of $\mbsb$, it holds that $z^s_{i} < u_{i}(B)$ for all $i \in \{1,2,3\}$ and $B \in \mbsb$, thus, $z^s_{i} < \min\{ u_{i}(B) : B \in \mbsb\}$ for all $i \in \{1,2,3\}$.
According to Lemma~\ref{lem:genredbox} and Corollary~\ref{cor:nonred}, redundancy occurs only for boxes $\hB, \tB \in \mbsb$ which are split with respect to the same component $i \in \{1,2,3\}$ (i.e., $u_{i}(\hB_{i})=u_{i}(\tB_{i})=z^s_{i}$) and for which  
$u_{l}(\hB_{i}) \geq u_{l}(\tB_{i})$ or $u_{l}(\hB_{i}) \leq u_{l}(\tB_{i})$ holds for all $l \neq i.$ By assumption, those boxes are removed, i.e., $\mbss$ contains only non-redundant boxes. 
We illustrate this case by the example depicted in Figure~\ref{fig:visualizelemma2}. 

\begin{figure}
\centering
\subfigure[Before the insertion of $z^2$]{
\includegraphics[width=.47\textwidth]{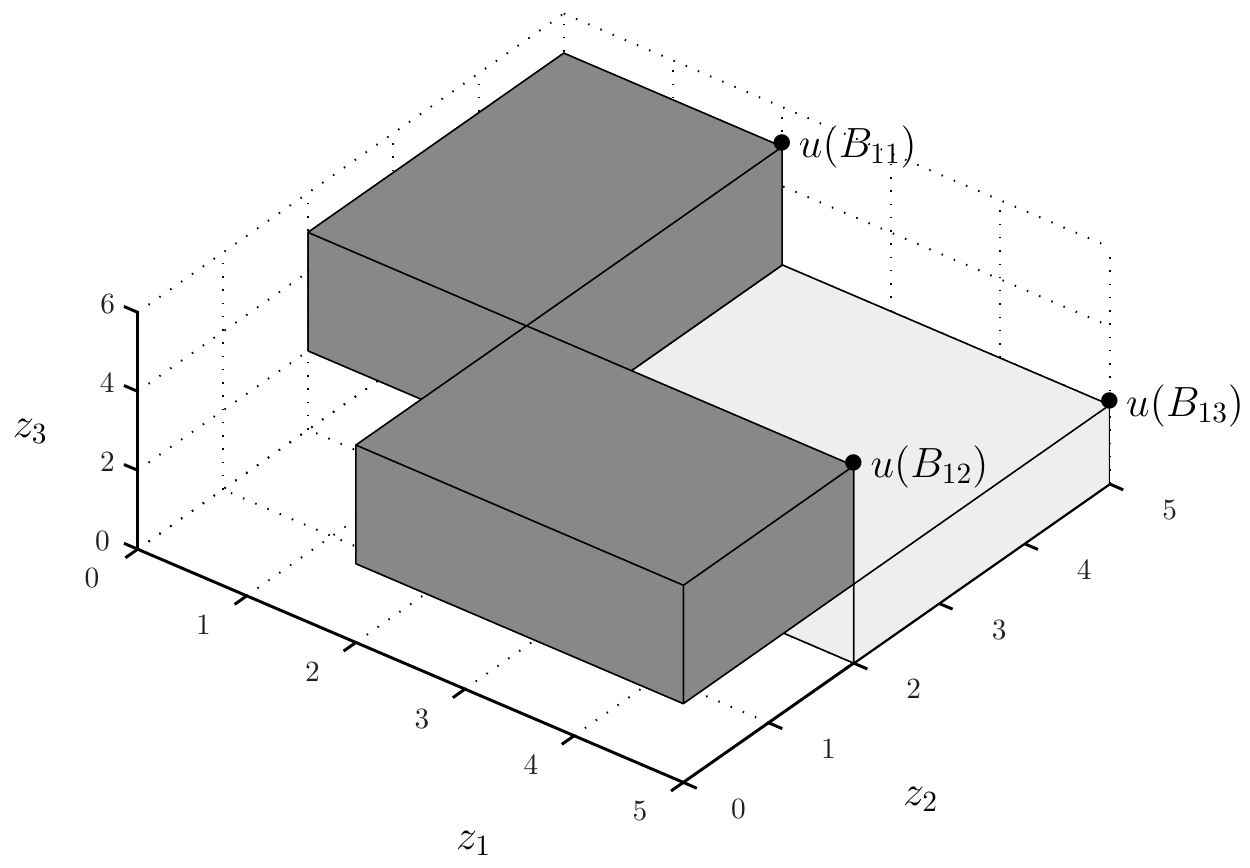}
}
\subfigure[After the insertion of $z^2$]{
\includegraphics[width=.47\textwidth]{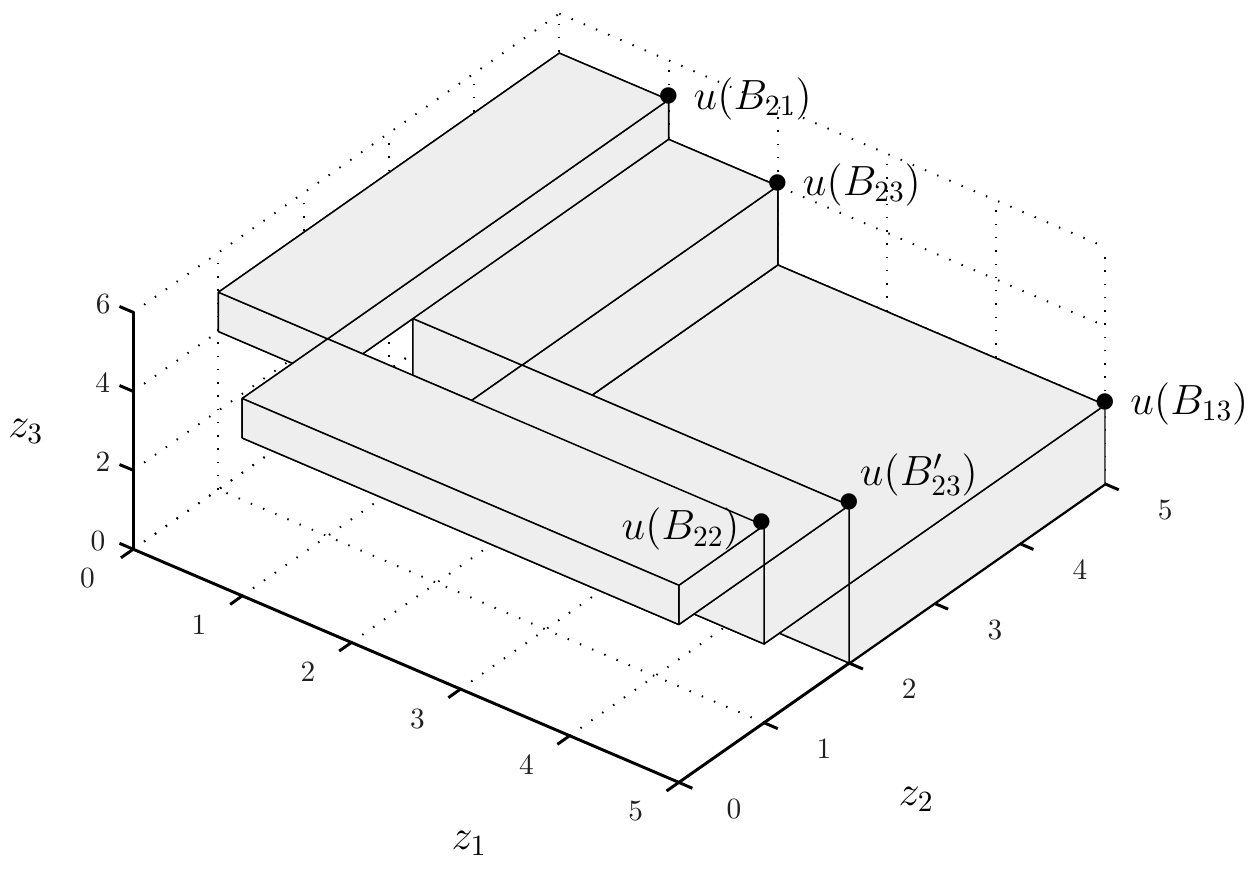} 
} 
\caption{Visualization of the upper bound vectors $u(B)$ in case~2 of the proof of Lemma~\ref{lem:neighb}: 
$z^2=(1,1,4)^{\top}$ lies in the two boxes $B_{11}$ with $u(B_{11})=(2,5,5)^{\top}$ and 
$B_{12}$ with $u(B_{12})=(5,2,5)^{\top}$, i.e.,  $|\overline{\mathcal{B}}_{2}| = 2$.
For a better illustration, the individual subsets $V(B)$ of all boxes are depicted. 
} 
\label{fig:visualizelemma2} 
\end{figure}

Let $\mbsb=\{\hB^1,\dots,\hB^P\}$ with $P \in \N, P \geq 2$.
The corresponding boxes in Figure~\ref{fig:visualizelemma2}~(a) are $\hB^1=B_{11}$ and $\hB^2=B_{12}$. 
For every $i \in \{1,2,3\}$, let $I_{i} \subseteq \{ 1,\dots,P \}$ be the index set of the boxes from $\mbsb$ whose split with respect to $i$ yields a non-redundant box. 
Note that $I_{i} \neq \emptyset$ for every $i \in \{1,2,3\}$, which can be seen as follows: Consider an arbitrary box $B \in \mbsb$. Applying the full $3$-split to $B$ results in three new boxes. Now any of the resulting boxes is removed if and only if there exists another box that dominates it. According to Lemma~\ref{lem:genredbox} the dominating box must have been created by a split with respect to the same component as the dominated box. Therefore, $I_{i} \neq \emptyset$ for every $i \in \{1,2,3\}$. We set $Q_{i}:=|I_{i}| \geq 1$ for every $i \in \{1,2,3\}$. Furthermore, let $\bar{u}_{i}:=\max\{u_{i}(B), B \in \mbsb \}$ for all $\itkk$ in the following, which is well defined as $\mbsb \neq \emptyset$. 

In the example depicted in Figure~\ref{fig:visualizelemma2}~(b), $Q_{1}=Q_{2}=1$ and $Q_{3}=2$. Moreover, from Figure~\ref{fig:visualizelemma2}~(a) we see that $\bar{u}=(5,5,5)^{\top}$. 

Consider now $i$ arbitrary but fixed.
Let $\hB^{I_{i}(1)},\dots,\hB^{I_{i}(Q_{i})}$ denote the boxes whose split with respect to component $i$ yields a non-redundant box. 
As $u_{i}(\hB_{i}^{I_{i}(q)})=z^s_{i}$ holds for all $\hB_{i}^{I_{i}(q)}\in \mbss, q=1,\dots,Q_{i}$, as $m=3$ and as we assume non-redundancy, we can order the boxes with respect to their upper bounds increasingly by some component $j \neq i$ and decreasingly by component $k \neq i,j$, i.e.,
\begin{align} 
z^s_{j} < \, &u_{j}(\hB_{i}^{I_{i}(1)}) < u_{j}(\hB_{i}^{I_{i}(2)}) < \dots < u_{j}(\hB_{i}^{I_{i}(Q_{i})}), \label{orderincr}\\
&u_{k}(\hB_{i}^{I_{i}(1)}) >u_{k}(\hB_{i}^{I_{i}(2)}) > \dots > u_{k}(\hB_{i}^{I_{i}(Q_{i})}) > z^s_{k}. \label{orderdecr}
\end{align} 
Thereby, $u_{j}(\hB_{i}^{I_{i}(Q_{i})})=u_{j}(\hB^{I_{i}(Q_{i})})=\bar{u}_{j}$ holds, since in the other case, i.e., if there was some $\tB \in \mbsb$ with $u_{j}(\tB)>u_{j}(\hB_{i}^{I_{i}(Q_{i})})$, either $\tB$ would have been the last box in \eqref{orderincr} with index $I_{i}(Q_{i})$ or $\hB_{i}^{I_{i}(Q_{i})}$ would have been dominated by $\tB$, both in contradiction to the construction. Analogously, 
$u_{k}(\hB_{i}^{I_{i}(1)})=u_{k}(\hB^{I_{i}(1)})=\bar{u}_{k}$ must hold. 

In the example depicted in Figure~\ref{fig:visualizelemma2}~(a), consider $i=3$ and, without loss of generality, let
$\hB^{I_{3}(1)}=B_{11}$ and $\hB^{I_{3}(2)}=\hB^{I_{3}(Q_{3})}=B_{12}$. 
The upper bounds  
$u(B_{11})=(2,5,5)^{\top}$ and 
$u(B_{12})=(5,2,5)^{\top}$ can be ordered increasingly with respect to component $j=1$ and decreasingly with respect to component $k=2$. It holds that
$u_{1}(\hB_{3}^{I_{3}(Q_{3})})=5=\bar{u}_{1}$ and $u_{2}(\hB_{3}^{I_{3}(1)})=5=\bar{u}_{2}$. 

If $u_{i}(\hB^{I_{i}(1)})=\max\{u_{i}(B) : B \in \mbs, u_{k}(B)=\bar{u}_{k}\}=:\bar{u}_{i,k}$ holds, 
then the split of $\hB^{I_{i}(1)}$ with respect to $j$ generates a non-redundant box, too, and, depending on the chosen enumeration, $I_{i}(1)$ either equals $I_{j}(1)$ or $I_{j}(Q_{j})$. 
Without loss of generality, we can set $I_{i}(1)=I_{j}(1)$. Otherwise, i.e., if 
$u_{i}(\hB^{I_{i}(1)})<\bar{u}_{i,k}$ holds, then
$\hB^{I_{i}(1)}_{j}$ is dominated by a unique box $\tB \in \mbsb$ with $u_{k}(\tB)=\bar{u}_{k}$ and $u_{i}(\tB)=\bar{u}_{i,k}$. Then $\tB= \hB^{I_{j}(1)}$ holds.

Analogously, if $u_{i}(\hB^{I_{i}(Q_{i})})=\max\{u_{i}(B) : B \in \mbs, u_{j}(B)=\bar{u}_{j}\}=:\bar{u}_{i,j}$ holds, 
then the split of $\hB^{I_{i}(Q_{i})}$ with respect to $k$ generates a non-redundant box, too, and, without loss of generality, we can identify $I_{i}(Q_{i})=I_{k}(Q_{k})$. 
 Otherwise, i.e., if $u_{i}(\hB^{I_{i}(Q_{i})})<\bar{u}_{i,j}$ holds, then
 $\hB^{I_{i}(Q_{i})}_{k}$ is dominated by a unique box $\tB \in \mbsb$ with $u_{j}(\tB)=\bar{u}_{j}$ and
 $u_{i}(\tB)=\bar{u}_{i,j}$. Then $\tB= \hB^{I_{k}(Q_{k})}$ holds.
 
Note that if $Q_{i}=1$, then 
$\hB^{I_{i}(1)}=\hB^{I_{i}(Q_{i})}=:\hB$ 
and $u_{j}(\hB)=\bar{u}_{j}$ as well as $u_{k}(\hB)=\bar{u}_{k}$ hold. 
In this case, $u_{i}(\hB)<\bar{u}_{i}$ must be satisfied, as otherwise $\hB$ would dominate any other box in $\mbsb$, a contradiction to $|\mbsb| > 1$ and $\mbs$ being non-redundant.

In the example depicted in Figure~\ref{fig:visualizelemma2}, consider $i=3$ and $k=2$. It holds that  
$u_{3}(\hB^{I_{3}(1)})=5=\max\{u_{3}(B) : B \in \mbs, u_{2}(B)=5 \}$.
The split of $\hB^{I_{3}(1)}$ with respect to $j=1$ generates the non-redundant box $B_{21}$.
If we consider $i=1$, then 
$u_{1}(\hB^{I_{1}(1)})=u_{1}(B_{11})=2<5=\max\{u_{1}(B) : B \in \mbs, u_{3}(B)=5 \}$, hence, the split of $B_{11}$ with respect to component $j=2$ must be redundant, and, indeed, the resulting box is dominated by $B_{22}$. 
\\
  
Analogously to Case~1, we will now indicate the neighbor boxes explicitly. Therefore, consider $\hB_{i}^{I_{i}(q)} \in \mbss$ for fixed $i \in \{1,2,3\}, q\in \{1,\dots,Q_{i}\}$. It holds that 

\begin{itemize}

\item[(i)] $\Bssi(\hB_{i}^{I_{i}(q)})=\Bsi(\hB^{I_{i}(q)})$:\\
Assume $\Bsi(\hB^{I_{i}(q)}) \in \mbsb$. By definition of $\Bsi$, $u_{i}(\Bsi(\hB^{I_{i}(q)}))<u_{i}(\hB^{I_{i}(q)})$ and $u_{l}(\Bsi(\hB^{I_{i}(q)})) \geq u_{l}(\hB^{I_{i}(q)})$ for all $l \neq i$ hold. But then, by an $i$-split of $\hB^{I_{i}(q)}$ and $\Bsi(\hB^{I_{i}(q)})$, the box $\hB^{I_{i}(q)}_{i}$ would be redundant. Therefore, $\Bsi(\hB^{I_{i}(q)}) \notin \mbsb$ must hold. Analogously to Case~1(i), we obtain $\Bssi(\hB_{i}^{I_{i}(q)})=\Bsi(\hB^{I_{i}(q)})$.

In the example depicted in Figure~\ref{fig:visualizelemma2}, it holds that
$B_{3}^3(B_{23})=B_{3}^2(B_{11})=B_{13}$ and
$B_{3}^3(B_{23}')=B_{3}^2(B_{12})=B_{13}$. 

\item[(ii)] Determination of $\Bssj(\hB_{i}^{I_{i}(q)})$ and $\Bssk(\hB_{i}^{I_{i}(q)})$ for $j,k \neq i$:\\
Consider all $\hB_{i}^{I_{i}(q)} \in \mbss, q=1,\dots,Q_{i}$, ordered as in \eqref{orderincr} and \eqref{orderdecr}: It holds that 
$$\Bssj(\hB_{i}^{I_{i}(q)})=\hB_{i}^{{I_{i}(q-1)}} \quad \text{for all} \; q=2,\dots,Q_{i},$$
as for all other boxes $\hB_{i}^{I_{i} (p)}, p \neq q-1$, either 
$u_{j}(\hB^{I_{i}  (p)}_{i}) < u_{j}(\hB^{{I_{i}(q-1)}}_{i})$ or \linebreak $u_{j}(\hB^{I_{i}{ (p)}}_{i}) > u_{j}(\hB^{I_{i}(q)}_{i})$ holds.  
Moreover, all new boxes split with respect to $j$ have component $u_{j}$ smaller than $u_{j}(\hB_{i}^{I_{i}(q-1)})$ and all new boxes split with respect to $k$ have component $u_{k}$ smaller than $u_{k}(\hB_{i}^{I_{i}(q)})$, and, thus, do not satisfy \eqref{neighk}. For all boxes $B \notin \mbsb$, it holds that $u_{l}(B)<\min\{ u_{l}(B) : B \in \mbsb\}$ for some $l$, so either $u_{j}(B)<u_{j}(\hB_{i}^{I_{i}(q-1)})$ or \eqref{neighk} is not satisfied.

Next, we determine $\Bssj(\hB_{i}^{I_{i}(1)})$: As $u_{j}(\hB_{i}^{I_{i}(q)})>u_{j}(\hB_{i}^{I_{i}(1)})$ for all $q=2,\dots,Q_{i}$, no box split with respect to $i$ can be the neighbor $\Bssj(\hB_{i}^{I_{i}(1)})$. Furthermore, as $u_{k}(\hB_{i}^{I_{i}(1)})>z^s_{k}$, $\Bssj(\hB_{i}^{I_{i}(1)})$ can not be found among the new boxes split with respect to $k$. 
Therefore, $\Bssj(\hB_{i}^{I_{i}(1)})$ can only be found among the boxes split with respect to component $j$. 
Now, as shown above, $u_{k}(\hB^{I_{i}(1)})=\bar{u}_{k}
$ holds, which implies that $u_{k}(\Bssj(\hB_{i}^{I_{i}(1)}))=\bar{u}_{k}$ must be satisfied. Therefore, the unique candidate for $\Bssj(\hB_{i}^{I_{i}(1)})$ is $\hB_{j}^{I_{j}(1)}$, which, as explained above, either equals the box obtained from $\hB^{I_{i}(1)}$ by a split with respect to $j$ or the unique box dominating it. 

Analogously, it can be shown that
$$\Bssk(\hB_{i}^{I_{i}(q)})=\hB_{i}^{I_{i}(q+1)} \quad \text{for all} \; q=1,\dots,Q_{i}-1,$$ and 
$$\Bssk(\hB_{i}^{I_{i}(Q_{i})})=\hB_{k}^{I_{k}(Q_{k})},$$ 
where $\hB_{k}^{I_{k}(Q_{k})}$ either equals the box obtained from $\hB^{I_{i}(Q_{i})}$ by a split with respect to $k$ (then $I_{i}(Q_{i})=I_{k}(Q_{k})$) or the unique box dominating it. 

In the example depicted in Figure~\ref{fig:visualizelemma2}, 
$B_{1}^3(B_{23}')=B_{23}$ and
$B_{1}^3(B_{23})=B_{21}$ hold.

\end{itemize} 

Finally, for all $B \notin \mbsb$ we obtain the following results which are equivalent to Case~1:
\begin{itemize}
\item[(iii)] If $\Bsi(B) \notin \mbsb$ for some $i \in \{1,2,3\}$, then $\Bssi(B)$ remains unchanged.

As in the example depicted in Figure~\ref{fig:visualizelemma2} box $B_{13}$ is the unique box which is not split and all of its neighbors are split, this case does not occur. 

\item[(iv)] If $\Bsi(B) =:\hB \in \mbsb$ for some $i \in \{1,2,3\}$, 
then, following the same argumentation as in Case~1(iv), $z^s_{j}>u_{j}(B)$ for one unique index $j \neq i$ and, thus, the correct candidate for $\Bssi(B)$ would be $\hB_{j}.$ 
It remains to show that $\hB_{j}$ exists and that 
$u_{i}(\hB_{j})=\max\{u_{i}(\tB) : \tB \in \mbss, u_{i}(\tB) < u_{i}(B) \}$. 

Assume that $\hB_{j}$ does not exist, i.e., it is redundant in $\mbss$. Then there exists $\bB \in \mbsb$ with $u_{i}(\bB) \geq u_{i}(\hB)$ and $u_{k}(\bB) \geq u_{k}(\hB).$ 
As $\bB, \hB \in \mbs$ and $\mbs$, by induction, is non-redundant, $u_{j}(\bB) < u_{j}(\hB)$ must hold. As $\bB \in \mbsb$ it follows that $z^s_{j}<u_{j}(\bB),$ so $u_{j}(B)<z^s_{j}<u_{j}(\bB)$ and $u_{k}(B)=u_{k}(\hB) \leq u_{k}(\bB)$ hold.
 If $u_{i}(\bB) \geq u_{i}(B),$ $B$ would have been redundant in $\mbs.$ Thus, $u_{i}(\bB) < u_{i}(B)$ must hold. However, as $\Bsi(B)=\hB$, the induction hypothesis then implies that $u_{i}(\bB)<u_{i}(\hB),$ a contradiction to the assumption on $\bB$.
Thus, $\hB_{j}$ is non-redundant and  
$u_{i}(\hB_{j})=u_{i}(\hB)=\max\{u_{i}(\tB) : \tB \in \mbss, u_{i}(\tB) < u_{i}(B) \}$ 
holds.

In the example depicted in Figure~\ref{fig:visualizelemma2}, consider $B^2_{1}(B_{13})=B_{11}$. 
Since $j=3$ is the unique index $\neq 1$ such that $z^2_{j}>u_{j}(B_{13})$, it holds that $B^3_{1}(B_{13})=B_{23}$, i.e., the new neighbor is the box which results from $B_{11}$ by a split with respect to~$j$. 
\end{itemize}
\qed 
\end{proof}
In the following corollary we summarize the properties of the neighbors of all new boxes obtained in the constructive proof of Lemma~\ref{lem:neighb}.
 
\begin{corollary} \label{cor:proof} 
Let Assumption~\ref{assump:znew2} be satisfied.
For every $i \in \{1,2,3\}$, let $I_{i} \subseteq \{ 1,\dots,P \}$, $I_{i} \neq \emptyset$, $P \in \N$, $|I_{i}|=Q_{i}$, be the index set of the boxes of $\mbsb$ whose split with respect to $i \in \{1,2,3\}$ yields a non-redundant box. 
Then for all new boxes $\hB^{I_{i}(q)}_{i}$, $q=1,\dots,Q_{i}$, it holds that
\begin{align}
\Bssi(\hB_{i}^{I_{i}(q)})&= \phantom{\;} \Bsi(\hB^{I_{i}(q)}) \quad \forall \, q=1,\dots,Q_{i}, \label{eq:neighi} \\
\Bssj(\hB_{i}^{I_{i}(q)}) &= 
\left\{
\begin{array}{ll}
\hB_{j}^{I_{j}(1)} & q=1,\\
\hB_{i}^{I_{i}(q-1)} & \forall \, q=2,\dots,Q_{i}, 
\end{array}
\right.  \label{eq:neighj}
\\
\Bssk(\hB_{i}^{I_{i}(q)}) &= 
\left\{
\begin{array}{ll}
\hB_{i}^{I_{i}(q+1)}  & \forall \, q=1,\dots,Q_{i}-1, \\ 
\hB_{k}^{I_{k}(Q_{k})} & q=Q_{i},
\end{array}
\right. \label{eq:neighk}
\end{align}
where the indices $j$ and $k$ are chosen as in \eqref{orderincr} and \eqref{orderdecr}, i.e., 
the boxes $\hB^{I_{i}(q)}_{i}, q=1,\dots,Q_{i},$ are ordered with respect to their upper bounds increasingly by component $j \neq i$ and decreasingly by component $k \neq i,j$.
Moreover, $I_{j}(1)$ and $I_{k}(Q_{k})$ are chosen such that $\hB_{j}^{I_{j}(1)}$ and $\hB_{k}^{I_{k}(Q_{k})}$ either equal $\hB_{j}^{I_{i}(1)}$ and $\hB_{k}^{I_{i}(Q_{i})}$, respectively, or the unique box dominating it. 
\end{corollary}
Using Lemma~\ref{lem:neighb} we can derive an explicit formulation of the individual subsets~$V(B)$ for $m=3$:
\begin{lemma} \label{lem:Vboxrep}
Let Assumption~\ref{assump:znew2} hold. Then, for $m=3$, the individual subsets $V(B), B \in \mbs$, which are introduced in Definition~\ref{def:nonred}, can be represented as
$$V(B)=\{ z \in B_{0} : v(B) \leqq z < u(B)\}$$ 
with  
\begin{equation} \label{eq:repv}
 v_{i}(B):= \left\{
\begin{array}{ll}
u_{i}(\Bsi(B)), & \text{if} \; \Bsi(B) \neq \emptyset, \\
z_{i}^I, & \text{otherwise} 
\end{array}
\right. , \quad i \in \{1,2,3\}.
\end{equation}
\end{lemma} 
\begin{proof}
For $\bB \in \mbs, s \geq 1$, by definition,
$$
V(\bB):=\bB \; \backslash \left( \bigcup_{\tB \in \mathcal{B}_{s}\backslash\{\bB\}} \tB  \right)
$$
holds. We consider the sets $\mathcal{B}_{s,i}:=\{ B \in \mbs : u_{i}(B) < u_{i}(\bB) \}$ for $i = 1,2,3$. For fixed $i \in\{1,2,3\}$, the following two cases can occur:
If $\mathcal{B}_{s,i} \neq \emptyset$, then, as shown in  Lemma~\ref{lem:neighb}, $\Bsi(\bB) \neq \emptyset$ and $\Bsi(\bB) \in \mathcal{B}_{s,i}$, where $u_{i}(\Bsi(\bB))=\max \{ u_{i}(B) : B \in\mathcal{B}_{s,i}\}$. 
Furthermore, as $u_{l}(\Bsi(\bB)) \geq u_{l}(\bB)$ for all $l \neq i$, 
$$
\bB \; \backslash \left( \bigcup_{\tB \in \mathcal{B}_{s,i}} \tB \right)
= \{z \in \bB : z_{i} \geq u_{i}(\Bsi(\bB)) \}.
$$
Otherwise, i.e., if $\mathcal{B}_{s,i} = \emptyset$, then, obviously, 
$\bB \; \backslash \left( \bigcup_{\tB \in \mathcal{B}_{s,i}} \tB \right)=\bB$. 
So, in both cases, it holds that
$$
\bB \; \backslash \left( \bigcup_{\tB \in \mathcal{B}_{s,i}} \tB \right)
= \{z \in \bB : z_{i} \geq v_{i}(\bB) \}$$
with 
$$
 v_{i}(\bB):= \left\{
\begin{array}{ll}
u_{i}(\Bsi(\bB)), & \text{if} \; \Bsi(\bB) \neq \emptyset, \\
z_{i}^I, & \text{otherwise.} 
\end{array}
\right.
$$
As every box $B \in \mbs \backslash \{ \bB \}$ belongs, due to the assumption of non-redundancy, at least to one set $\mathcal{B}_{s,i}, i \in \{1,2,3\},$ there does not exist any other box which can reduce $V(\bB)$ further. Thus,  we obtain the desired representation.
\qed 
\end{proof}
Lemma~\ref{lem:Vboxrep} shows that for $m=3$ the individual subset of a box can be represented as a box itself. As the upper bound of $V(B)$ and $B$ are the same, $V(B)$ can be described by its lower bound $v(B) \in \R^m$ only. 
Next we show, using Corollary~\ref{cor:proof},
how the lower bounds $v(B)$ can be updated in an iterative algorithm.
\begin{lemma} \label{lem:updateV}
Let Assumption~\ref{assump:znew2} be satisfied. We use the notation of Corollary~\ref{cor:proof}.
Let $\hB^{I_{i}(q)}_{i}, q=1,\dots,Q_{i}$, be the non-redundant boxes obtained from $\hB^{I_{i}(q)} \in \mbsb$ by a split with respect to $i \in \{1,2,3\}$.
Then the lower bound vectors $v(B) \in \R^m$ of these new boxes in $\mbss$ are determined by
\begin{align*}
v_{i}(\hB_{i}^{I_{i}(q)})&= \phantom{\;} v_{i}(\hB^{I_{i}(q)}) \quad \forall \, q=1,\dots,Q_{i},\\
v_{j}(\hB_{i}^{I_{i}(q)}) &=
\left\{
\begin{array}{ll}
u_{j}(\hB_{j}^{I_{j}(1)})=z^s_{j} \,\, & q=1,\\
u_{j}(\hB_{i}^{I_{i}(q-1)}) & \forall \, q=2,\dots,Q_{i},
\end{array}
\right.  
\\
v_{k}(\hB_{i}^{I_{i}(q)}) &=
\left\{
\begin{array}{ll}
u_{k}(\hB_{i}^{I_{i}(q+1)})  & \forall \, q=1,\dots,Q_{i}-1, \\ 
u_{k}(\hB_{k}^{I_{k}(Q_{k})})=z^s_{k} & q=Q_{i}.
\end{array}
\right.
\end{align*}
All individual subsets $V(B)$ of all $B \notin \mbsb$ remain unchanged.
\end{lemma} 
\begin{proof}
The update of $v(\hB^{I_{i}(q)}_{i})$ of all new boxes $\hB^{I_{i}(q)}_{i}, q=1,\dots,Q_{i}$, for some fixed $i \in \{1,2,3\}$ is derived directly from Corollary~\ref{cor:proof}. 
The individual subsets of all boxes which are not split in the current iteration do not change, as, according to the proof of Lemma~\ref{lem:neighb},
either $\Bssi(B)$ remains unchanged (Case~(iii)) or $\Bssi(B)=\hB_{j}$ (Case~(iv)), i.e., $u_{i}(B)$ remains unchanged.
\qed 
\end{proof}
Recall that we want to split a box $B \in \mbsb$ with respect to a component $\itkk$ if and only if the individual subset $V(B_{i})$ of the resulting box $B_{i}$ is non-empty, which is equivalent to $B_{i}$ being non-redundant. With the vector $v(B) \in \R^m$ at hand, this can be easily checked, as the following lemma shows.

\begin{lemma} \label{lem:vsplitcorrect} 
Let Assumption~\ref{assump:znew2} hold up to iteration $s-1$ for $s\geq 2$, i.e., let $\mbs$ be a correct, non-redundant decomposition of the search region obtained by iterative
$3$-splits. 
Let $z^s \in \NS$ satisfy Assumption~\ref{assump:znew2}~(1), and let $B_{i}$ be the box obtained from $B \in \mbsb$ by a split with respect to component $i \in \{1,2,3\}$. Then $B_{i}$ is non-redundant if and only if 
$z^s_{i} > v_{i}(B)$ holds. 
\end{lemma}
\begin{proof}
Consider a fixed $i \in \{1,2,3\}$.
\begin{itemize}
\item[``$\Rightarrow$'':] Let $B_{i}$ be non-redundant and assume that $z^s_{i} <v_{i}(B)$ holds. (The case $z^s_{i} = v_{i}(B)$ does not occur due to Assumption~\ref{assump:znew2}~(1).)
Then $v_{i}(B) > z_{i}^I$, and, thus, $v_{i}(B)=u_{i}(\Bsi(B))$ with $\Bsi(B) \neq \emptyset$. As $u_{l}(\Bsi(B)) \geq u_{l}(B)$ for all $l \neq i$, $z^s \in \Bsi(B)$ must hold. But then, $B_{i}$ would be redundant as it would be dominated by the box obtained from $\Bsi(B)$ by a split with respect to $i$, a contradiction to the assumption of non-redundancy.
\item[``$\Leftarrow$'':] Let $z^s_{i} > v_{i}(B).$ A split of $B$ with respect to $i$ yields $B_{i}=\{ z \in B : z_{i} <z_{i}^s \}.$ 
Assume that there exists $\tB_{i} \neq B_{i}$ which dominates $B_{i}$. 
As $\mbs$ is non-redundant and due to Lemma~\ref{lem:genredbox}, $\tB_{i}$ must result from a split with respect to $i$ from some box $\tB \in \mbsb,$ i.e., $z^s \in \tB$ must hold. 
As $B$ and $\tB$ are split with respect to $i$, $u_{i}(B_{i})=u_{i}(\tB_{i})=z^s_{i}$ holds, and, due to the assumption that $\tB_{i}$ dominates $B_{i}$, $u_{l}(\tB) \geq u_{l}(B)$ for all $l \neq i$. 
Now $B,\tB \in \mbs$ and $\mbs$ being non-redundant imply that $u_{i}(\tB)<u_{i}(B)$. 
This in turn means that $v_{i}(B) \geq u_{i}(\tB)$. 
But then $z^s_{i}>u_{i}(\tB),$ a contradiction to $z^s \in \tB$. 
It follows that $B_{i}$ is non-redundant. 
\end{itemize}
\end{proof}
Lemma~\ref{lem:vsplitcorrect} provides a tool for defining a split operation for tricriteria problems which generates all boxes that are necessary for maintaining the correctness of a decomposition, but avoids the generation of redundant boxes. 
We call the split based on the individual subsets $V(B)$ a $v$-split in the following.

\begin{definition}[$v$-split] \label{def:vsplit} 
Let Assumption~\ref{assump:znew2} hold up to iteration $s-1$ for $s\geq 2$, i.e., let $\mbs$ be a correct, non-redundant decomposition of the search region obtained by iterative $3$-splits, and let $z^s \in \NS$. 
We call the split of a box $B \in \mbsb$ with respect to components $i \in \{1,2,3\},$ for which
\begin{equation} \label{ineq:vsplit}
z^s_{i} \geq v_{i}(B)
\end{equation} 
holds, a $v$-split of $B$. 
\end{definition}
Note that equality in \eqref{ineq:vsplit} does not occur due to Assumption~\ref{assump:znew2}~(1). However, as Assumption~\ref{assump:znew2}~(1) will be removed in Section~\ref{subsec:generalize}, we present the $v$-split already at this point in this general form.
\begin{lemma}\label{lem:vsplitisnonredundant}
Let Assumption~\ref{assump:znew2}~(1),(2) hold. Then the iterative application of a $v$-split to every $B \in \mbsb$ in every iteration $s \geq 1$ yields a correct, non-redundant decomposition.
\end{lemma}
\begin{proof}
Due to Assumption~\ref{assump:znew2}~(1), $z^s_{i} \geq v_{i}(B)$ is equivalent to $z^s_{i} > v_{i}(B)$. According to Lemma~\ref{lem:vsplitcorrect}, the $v$-split avoids exactly the generation of redundant boxes and, therefore, yields a correct, non-redundant decomposition. 
\qed 
\end{proof}
%
%
\subsection{An algorithm for tricriteria problems based on the $v$-split}
Algorithm~\ref{table:algoVsplit} implements the $v$-split. 
As in Algorithm~\ref{table:algofullsplit}, an initial box $B_{0}$ is computed, which is represented by its upper bound $u(B_{0})$. 
Additionally, for $B_{0}$ as well as for all other boxes $B$ which are generated in the course of the algorithm, the lower bound of the individual subset $v(B)$ is saved. 
Analogously to Algorithm~\ref{table:algofullsplit}, as long as the decomposition contains unexplored boxes, a box is selected and a  subproblem is solved. 
If the problem is infeasible, the selected box is deleted from the list of unexplored boxes. Otherwise, the nondominated point $z^s$ is saved and all boxes are determined that contain $z^s$. 
Now, different from Algorithm~\ref{table:algofullsplit}, $z^s$ is compared component\-wise to $v(B)$ for every $B \in \mbsb$. A split with respect to component $i$ is performed if and only if $z^s_{i} \geq v_{i}(B)$ and $z^s_{i} > z_{i}^I$ hold. If $v_{i}(B)>z^s_{i}$ for all $i \in \{1,2,3\}$, then $B$ is deleted. Finally, the vectors $v$ of all new boxes are updated according to Lemma~\ref{lem:updateV} and a new iteration starts.

Note that according to the proof of Lemma~\ref{lem:neighb} we can order all newly generated, non-redundant boxes resulting from a split with respect to component $i$ such that their upper bound values $u$ are increasing in one component $j \neq i$ and decreasing in the remaining component $k\neq i,j$.
Hence, in Line~\ref{algo:linetest} in the procedure \textsc{UpdateIndividualSubsets}, strict inequalities hold between the components of each pair of upper bounds. 
However, in order to make the algorithm also applicable when Assumption~\ref{assump:znew2}~(1) is removed, see Section~\ref{subsec:generalize} below, we formulate Algorithm~\ref{table:algoVsplit} already in a general form. 
Therefore, in Line~\ref{algo:linetest}, the strict inequalities are replaced by inequalities and in Lines~\ref{algo:linetest2} to~\ref{algo:linetest3} the case that the upper bound vectors of two boxes are equal is handled. 
Note that this case does not occur under Assumption~\ref{assump:znew2}~(1). 

According to Lemma~\ref{lem:vsplitisnonredundant}, Algorithm~\ref{table:algoVsplit} maintains a correct, non-redundant decomposition in each iteration. 


\begin{algorithm}[t]
\caption{Algorithm with $v$-split for $m=3$} \label{table:algoVsplit}
\begin{algorithmic}[1]
\INPUT Image of the feasible set $Z \subset \R^m$, implicitly given by some problem formulation
\State $\NS:=\emptyset$; $\delta >0$; 
\State \Call{InitStartingBoxVsplit}{$Z,\delta$}; \Comment{Initialize starting box}
\State $s:=1$;
\While {$\mathcal{B}_{s} \neq \emptyset$} 
\State Choose $\bB \in \mbs;$ \label{algoV:select} \Comment{Select a box from the decomposition}
\State $z^{s}:=opt(Z,u(\bB))$; \Comment{Solve subproblem}
\If {$z^{s} = \emptyset$}  \Comment{Subproblem infeasible} \label{table:algoVsplitempty}
\State $\mbss:=\mbs \backslash \{ \bB\}$;   \Comment{Remove (empty) box} 
\Else
\State $\NS:=\NS \cup \{z^s\};$ \Comment{Save nondominated point}
\State $\mbss:=\mbs;$ \Comment{Copy set of current boxes}
\State \Call{GenerateNewBoxesVsplit}{$\mbs,z^s,z^I,\mbss$};
\State \Call{UpdateIndividualSubsets}{$\mathcal{S}_{1},\mathcal{S}_{2},\mathcal{S}_{3},\mbss$};
\EndIf
\State $s:=s+1$;
\EndWhile
\State \Return Set of nondominated points $\NS$
\algstore{bkbreak}
\end{algorithmic}
\end{algorithm}
\begin{algorithm} 
\begin{algorithmic}[1]
\algrestore{bkbreak}
\Procedure{InitStartingBoxVsplit}{$Z,\delta$}
\For {$j=1$ to $3$} \Comment{Compute bounds on $Z$}
\State $z^I_{j}:=\min \{z_{j} \, : \, z \in Z \}$;  
\State $z^M_{j}:=\max \{z_{j} \, : \, z \in Z \}+\delta$; 
\State $v_{j}(B_{0}):=z^I_{j}$; $u_{j}(B_{0}):=z^M_{j};$
\EndFor
\State $\mathcal{B}_{1}:=\{ B_{0}\}$; \Comment{Initialize set of boxes}
\State \Return $\mathcal{B}_{1}$
\EndProcedure
\Statex
\Procedure{GenerateNewBoxesVsplit}{$\mbs,z^s,z^I,\mbss$}
\State $\mathcal{S}_{i}:=\emptyset$, $i=1,2,3;$ 
\ForAll {$B \in \mbs$}
\If {$z^s < u(B)$}  \Comment{Point is contained in box}
\For {$i=1$ to $3$}  \Comment{Apply $v$-split}
\If {$ z^s_{i} \geq v_{i}(B)$ and $ z^s_{i} > z^I_{i}$}  
\State $u(B_{i}):=u(B); v(B_{i}):=v(B)$;  \Comment{Create a copy of $B$} 
\State $u_{i}(B_{i}):=z^s_{i};$ \Comment{Update upper bound} 
\State $\mathcal{S}_{i}:=\mathcal{S}_{i} \cup \{B_{i}\};$ \Comment{Save new box in respective set $\mathcal{S}_{i}$} 
\EndIf 
\EndFor
\State $\mbss:=\mbss \backslash \{ B\};$ \Comment{Remove $B$} 
\EndIf 
\EndFor
\State \Return $\mbss, \mathcal{S}_{1}, \mathcal{S}_{2}, \mathcal{S}_{3};$
\EndProcedure
\Statex
\Procedure{UpdateIndividualSubsets}{$\mathcal{S}_{1},\mathcal{S}_{2},\mathcal{S}_{3},\mbss$}
\For {$i=1$ to $3$} 
\State $Q:=|\mathcal{S}_{i}|;$ 
\State Sort all boxes $B_{i}^{I_{i}(q)}, q=1,\dots,Q_{i},$ in $\mathcal{S}_{i}$ such that for $j,k \neq i$ \label{algo:linetest}
\Statex \hspace{1.2cm} $u_{j}(B_{i}^{I_{i}(1)}) \leq u_{j}(B_{i}^{I_{i}(2)}) \leq \dots \leq u_{j}(B_{i}^{I_{i}(Q_{i})}) $ and
\Statex \hspace{1.2cm} $u_{k}(B_{i}^{I_{i}(1)}) \geq u_{k}(B_{i}^{I_{i}(2)}) \geq \dots \geq u_{k}(B_{i}^{I_{i}(Q_{i})})$;
\If {$u(B_{i}^{I_{i}(q)}) = u(B_{i}^{I_{i}(q+1)})$ for some $q=1,\dots,Q_{i}-1$} \label{algo:linetest2}
\State (re)sort $B_{i}^{I_{i}(q)}$ and $B_{i}^{I_{i}(q+1)}$ such that 
\Statex \hspace{1.5cm} $v_{j}(B_{i}^{I_{i}(q)}) \leq v_{j}(B_{i}^{I_{i}(q+1)})$ and 
 $v_{k}(B_{i}^{I_{i}(q)}) \geq v_{k}(B_{i}^{I_{i}(q+1)})$; 
\EndIf \label{algo:linetest3} 
\State Set $v_{j}(B_{i}^{I_{i}(1)}):=z^s_{j}$; $v_{k}(B_{i}^{I_{i}(Q_{i})}):=z^s_{k};$ \Comment{Update $v$}
\For {$q=2$ to $Q_{i}$} 
\State $v_{j}(B_{i}^{I_{i}(q)}):=u_{j}(B_{i}^{I_{i}(q-1)})$; $v_{k}(B_{i}^{I_{i}(q-1)}):=u_{k}(B_{i}^{I_{i}(q)});$ 
\EndFor 
\State $\mbss:=\mbss \cup \mathcal{S}_{i} ;$ \Comment{Append new boxes}
\EndFor
\State \Return $\mbss$
\EndProcedure
\end{algorithmic}
\end{algorithm}


\begin{example}[Application of Algorithm~\ref{table:algoVsplit}] \label{ex:basicvsplit}
Consider again the tricriteria problem of Example~\ref{ex:split} with initial search region 
$B_0:=\{z \in Z \, : \, 0 \leq z_{i} \leq 5 \; \forall \, i=1,2,3 \}$
and $V(B_{0})=B_{0}$, thus, $v(B_{0})=(0,0,0)^{\top}$. 
Consider $z^{1}=(2,2,2)^{\top}$. The $v$-split applied to the initial box equals a full $3$-split and, thus, results in
$$
B_{1,i}:=\{z \in B_{0} \, : \, z_{i} < 2 \}, \; i=1,2,3. 
$$
The corresponding individual subsets are
$$
V(B_{1,i}):=\{z \in B_{1,i} \, : \, z_{j} \geq 2 \; \forall \, j \neq i\}, \; i=1,2,3, 
$$
thus, $v(B_{11})=(0,2,2)^{\top}$, $v(B_{12})=(2,0,2)^{\top}$ and $v(B_{13})=(2,2,0)^{\top}$. 
Let $z^{2}=(1,1,4)^{\top}$.  
It holds that $z^{2} \in B_{11}$ as well as $z^{2} \in B_{12}$, but $z^{2} \notin B_{13}$. 
Consider first the $v$-split in $B_{11}$: As $z^2_{1} \geq v_{1}(B_{11})$, $z^2_{2} \not\geq v_{2}(B_{11})$ and $z^2_{3}\geq v_{3}(B_{11}),$ $B_{11}$ is split with respect  to the first and third component into  
$$
B_{21}  :=\{z \in B_{11} \, : \, z_{1} < 1\}  \quad \text{and} \quad 
B_{23}  :=\{z \in B_{11} \, : \, z_{3}<4 \}
$$
and $\mathcal{S}_{1}=\{ B_{21}\}$, $\mathcal{S}_{2}=\emptyset$ and $\mathcal{S}_{3}=\{ B_{23}\}$.
Applying the $v$-split to $B_{12}$ results in a split with respect to the second and third component into
$$
B_{22}  :=\{z \in B_{12} \, : z_{2}<1\} \quad \text{and} \quad 
B_{23}'  :=\{z \in B_{12}\, : \, z_{3}<4 \} 
$$
and $\mathcal{S}_{1}=\{ B_{21}\}$, $\mathcal{S}_{2}=\{ B_{22}\}$ and $\mathcal{S}_{3}=\{ B_{23},B_{23}'\}$.
Note that the redundant boxes which were obtained with the full $3$-split in Example~\ref{ex:split} are not generated by the $v$-split.

Finally, the individual subsets of the new boxes of each set $\mathcal{S}_{i}, i \in \{1,2,3\}$, are updated: Box $B_{21}$ is the only box generated for $i=1$, box $B_{22}$ the only one for $i=2$. Therefore,
$v(B_{21})=(v_{1}(B_{11}),z^2_{2},z^2_{3})^{\top}=(0,1,4)^{\top}$ and
$v(B_{22})=(z^2_{1},v_{2}(B_{12}),z^2_{3})^{\top}=(1,0,4)^{\top}$.
Boxes $B_{23}$ and $B_{23}'$ are both generated by a split with respect to the third component $i=3$. We can order the upper bounds of the boxes $u(B_{23})=(2,5,4)^{\top}$ and $u(B_{23}')=(5,2,4)^{\top}$ increasingly with respect to component $j=1$ 
and, at the same time, decreasingly with respect to $k=2$,
thus, $B_{3}^{I(1)}:=B_{23}$ and $B_{3}^{I(2)}:=B_{23}'$ and then set
$$
\begin{array}{ll}
v_{1}(B_{23})=z^2_{1}=1, & \quad v_{2}(B_{23}')=z^2_{2}=1, \\
v_{2}(B_{23})=u_{2}(B_{23}')=2,  & \quad v_{1}(B_{23}')=u_{1}(B_{23})=2.
\end{array}
$$
The third component is not changed, so $v(B_{23})=(1,2,2)^{\top}$ and $v(B_{23}')=(2,1,2)^{\top}$. 
\end{example}


\subsection{A linear bound on the number of subproblems for $m=3$}

In the following, we will bound the number of boxes generated in the course of Algorithm~\ref{table:algoVsplit}. 
If a box $B \in \mbs$ contains the current point $z^s$, i.e., if $B \in \mbsb$, then we can make the following assertion concerning the neighbors of $B$:
\begin{lemma} \label{lem:splitneighbor}
Let $s \geq 1$ be an iteration of Algorithm~\ref{table:algoVsplit} in which a nondominated point is generated, i.e., $z^s \neq \emptyset$. 
Consider any $B \in \mbsb$. 
We denote by $J_{B} \subseteq \{1,2,3\}$ the index set of all components with respect to which $B$ is split, and by $\overline{J}_{B}:=\{1,2,3\} \backslash J_{B}$ the complement of $J_{B}$. 
Then the following holds:
\begin{enumerate}
\item If $\JB \neq \emptyset$, then for every $j \in \JB$, the neighbor $\Bsj(B)$ exists  and contains $z^s$, i.e., $\Bsj(B) \neq \emptyset$ and $\Bsj(B) \in \mbsb$ holds for every $j \in \JB$.
\item If $\JB = \emptyset$, then $\mbsb=\{B\}$ holds.
\end{enumerate} 
\end{lemma}
\begin{proof}
Let $B \in \mbsb$. By definition of the $v$-split, it holds that $z^s < u(B), z^s_{j} \geq v_{j}(B)$ for every $j \in J_{B}$ and $z^s_{j} < v_{j}(B)$ for every $j \in \JB$. Thus, $v_{j}(B)>z^I_{j}$ holds for every $j \in \JB$. This, however, implies that $\Bsj(B) \neq \emptyset$ for every $j \in \JB$, see the update of $v$ in \eqref{eq:repv}.
  
First, let $\JB \neq \emptyset$. Then, for fixed $j \in \JB$ and according to Definition~\ref{def:neighbors}, $u_{j}(\Bsj(B)) \leq u_{j}(B)$ and $u_{l}(\Bsj(B)) \geq u_{l}(B)$ for all $l \neq j$. 
As $u_{j}(\Bsj(B)) = v_{j}(B)>z^s_{j},$ $\Bsj(B) \in \mbsb$ holds.

Now, consider the case $\JB = \emptyset$. 
Then, due to Lemma~\ref{lem:neighb}, every box $\tB \in \mbs\backslash\{B\}$ has upper bound $u_{l}(\tB) \leq v_{l}(B)$ for at least one $l \in \{1,2,3\}$. This implies that $z^s \notin \tB$ for any $\tB \in \mbs \backslash \{ B\},$ thus, $\mbsb=\{B\}$. 
\qed 
\end{proof}
 
With the help of Lemma~\ref{lem:splitneighbor}, we can bound the number of new boxes that are generated in each iteration of Algorithm~\ref{table:algoVsplit}. 

\begin{lemma} \label{lem:decompnonred}
In every iteration $s \geq 1$ of Algorithm~\ref{table:algoVsplit} in which a  nondominated point $z^s$ is found, i.e., $z^s \neq \emptyset$, the number of boxes in the decomposition increases by at most two.
\end{lemma}
\begin{proof}
If there exists a box $B \in \mbsb$ which is split with respect to all three components, then, using Lemma~\ref{lem:splitneighbor}, $\mbsb=\{B\}$ holds, thus, $|\mbsb|=1$. In this case, the box $B$ is removed and replaced by three new boxes in the decomposition, and, thus, the number of boxes in the decomposition increases by two.

It follows that if $|\mbsb|>1$, then every $B \in \mbsb$ is split with respect to at most two components.
Let $|\mbsb|>1$ and let $B \in \mbsb$ be split with respect to two components $i,j \in \{1,2,3\}$, $j \neq i$. Then, for all other boxes $\tB \in \mbs \backslash\{B\}$ it holds that $u_{l}(\tB) \leq v_{l}(B)$ for some $l \in \{1,2,3\}$. 
If $l=i$, then $u_{i}(\tB) \leq v_{i}(B) \leq z^s_{i}$, thus, the box is not split with respect to $i$. Analogously, if $l=j$, then $u_{j}(\tB) \leq v_{j}(B) \leq z^s_{j}$, thus, the box is not split with respect to $j$. 
If $l=k$ (with $k\neq i,j$), then, for any $\tB$ satisfying $u_{k}(\tB) \leq v_{k}(B)$ it holds that $v_i(\tB) \geq u_i(B)$ or $v_j(\tB) \geq u_j(B)$, thus, $\tB$ can not be split with respect to both components $i$ and $j$. 

Therefore, if two boxes are split with respect to two components, these components must differ in one component. This implies that in one iteration, at most three boxes are split with respect to two components. Any other boxes in $\mbsb$ are split with respect to at most one component. 

In case that three boxes are split with respect to two components, six new boxes would replace three old ones, thus, the number of boxes would increase by three. So it remains to show that in this case, at least one box $B$ in $\mbsb$ is removed without being split, i.e., $v(B)>z^s$ holds for at least one $B \in \mbsb$. In other words, we have to prove the existence of a '0-box', i.e., a box, which is contained in $\mbsb$, but is not split with respect to any component.

To this end, we assume to the contrary that $\mbsb$ contains three boxes which are split with respect to two components ('2-boxes'), respectively, but that no '0-box' exists. From Lemma~\ref{lem:splitneighbor} we see that
a '2-box' has exactly one neighbor in $\mbsb$, as $\JB$ contains exactly one index. A '1-box' has exactly two neighbors in $\mbsb$, while all three neighbors are contained in $\mbsb$ in case of a '0-box'.
Now, starting from a '2-box', one uniquely defined neighbor of it must be in $\mbsb$. If that box is also a '2-box' (see Figure~\ref{fig:branches} on the left), then no neighbor of the latter box is in $\mbsb$. The third '2-box' would require a neighbor in $\mbsb$, but only '1-boxes' are available, which require a second neighbor in turn. Thus, a fourth '2-box' would be needed, which, however, does not exist. Therefore, the three '2-boxes' must all be connected by one structure of neighbors. But this implies that there exists exactly one '0-box' connecting the three branches emerging from each '2-box' (see Figure~\ref{fig:branches} on the right).
\qed 
\end{proof}

\begin{figure}
\begin{minipage}{.35\textwidth}
\begin{center}
\includegraphics{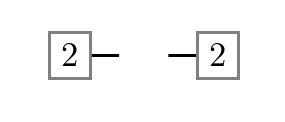}
\end{center} 
\end{minipage}
\begin{minipage}{.6\textwidth}
\begin{center}
\includegraphics{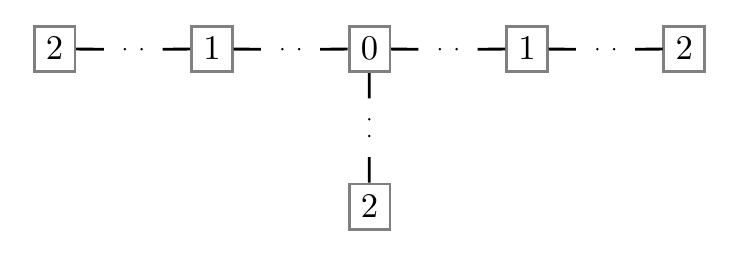}
\end{center}
\end{minipage}
\caption{Possible neighborhood structures of boxes in $\mbsb$. Left figure: $\mbsb$ contains two '2-boxes'; right figure: $\mbsb$ contains three '2-boxes'.} \label{fig:branches}
\end{figure}

\begin{theorem} \label{cor:runtime}
Let a problem with a finite set of nondominated points be given.  
After having computed the starting box based on the ideal point and a global upper bound on $Z$, Algorithm~\ref{table:algoVsplit} requires the solution of at most $3 |\NS|-2$ subproblems in order to generate the entire nondominated set $\NS$. 
\end{theorem}
\begin{proof}
In every iteration of Algorithm~\ref{table:algoVsplit}, one subproblem is solved. 
Thus, the number of subproblems to be solved equals the number of iterations. 
When a nondominated point is generated, the number of boxes increases by at most two according to Lemma~\ref{lem:decompnonred}. 
As every nondominated point is generated exactly once, and since every empty box is investigated exactly once in order to verify that no further nondominated points are contained, at most $3 |\NS|$ boxes are explored in the course of the algorithm. 
Together with the initial box, at most $3 |\NS|+1$ boxes are explored, which corresponds to the number of subproblems to be solved.
As we additionally assume that the ideal point is given, we can reduce this bound further: 
In every iteration in which the current nondominated point equals the ideal point in at least one component, one box per component equal to the ideal point can be directly discarded. For each component $\itkk$, there must exist at least one nondominated point whose $i$-th component equals $z^I_{i}$. Therefore, the total number of subproblems to be solved is at most $3|\NS|-2$.
\qed \end{proof}


\subsection{Applying the $v$-split to arbitrary nondominated sets for tricriteria problems} \label{subsec:generalize}
For the construction of the $v$-split we assumed that no pair of nondominated points has an identical value in at least one component, i.e., that all values are pairwise different (Assumption~\ref{assump:znew2}~(1)). Under this assumption, the individual subsets of all (non-redundant) boxes are boxes themselves, which is the basis for the $v$-split criterion.
In practice, Assumption~\ref{assump:znew2}~(1) may be violated since arbitrary nondominated points may coincide in up to $m-2$ components, i.e., in one component for $m=3$. In this case, additional redundant boxes may occur as the following example shows.

\begin{example}\label{ex:generalvsplit}
Let $z^{1}=(3,1,4)^{\top},$ $z^{2}=(3,2,1)^{\top}$ and let the initial search region be given as $B_0:=\{z \in Z \, : \, 0 \leq z_{i} \leq 5 \; \forall \, i=1,2,3 \}$.  
If we insert $z^1$ into $B_{0}$, we obtain the three subboxes
$B_{1,i}:=\{z \in B_{0} \, : \, z_{i} < z^1_{i} \}, \; i=1,2,3$, with respective upper bounds
$$u(B_{11})=(3,5,5)^{\top},
u(B_{12})=(5,1,5)^{\top},
u(B_{13})=(5,5,4)^{\top}.
$$
The second point $z^{2}=(3,2,1)^{\top}$ is only contained in $B_{13}$. Thus, $B_{13}$ is replaced by the three subboxes
$B_{2,i}:=\{z \in B_{13} \, : \, z_{i} < z^2_{i} \}, \; i=1,2,3$, with respective upper bounds 
$$u(B_{21})=(3,5,4)^{\top},
u(B_{22})=(5,2,4)^{\top},
u(B_{23})=(5,5,1)^{\top}.
$$ 
It holds that $B_{21} \subseteq B_{11}$.  
\end{example}
Note that under Assumption~\ref{assump:znew2}~(1) no redundancy appears if $|\mbsb|=1$, which is, as shown in Example~\ref{ex:generalvsplit}, no longer true for arbitrary nondominated points. 
If the redundant box $B_{21}$ is removed from the decomposition, i.e., if we set $\mathcal{B}_{3}:=\{B_{11},B_{12},B_{22},B_{23}\}$, then, however, the individual subset $V(B_{11})$ does not have the structure of a box anymore, as
\begin{align*}
V(B_{11})&:=B_{11} \; \backslash \left( \bigcup_{\tB \in \mathcal{B}_{3}\backslash\{B_{11} \}} \tB  \right)\\
&= \{ z \in B_{11} : z \geqq (0,2,1)^{\top} \} \cup \{ z \in B_{11} : z \geqq (0,1,4)^{\top} \}.
\end{align*}
Nevertheless, the box format of the individual subsets can be preserved if we maintain the redundant box $B_{21}$ in the decomposition. Then, $V(B_{11})=\{z \in B_{11} : z \geqq v(B_{11}) \}$ with $v(B_{11})=(0,1,4)^{\top}$ holds. Also the individual subset of $B_{21}$ has box format with $v(B_{21})=(3,2,1)^{\top}$. However, as $B_{21} \subseteq B_{11}$, it holds that $V(B_{21})=\emptyset$.

Despite some individual subsets being empty, the $v$-split can be applied regularly in the following iterations since the lower bound $v(B)$ is compared to the current nondominated point~$z^s$ component-wise. 
Thus, it is irrelevant for the $v$-split whether $V(B)$ for some $B \in \mbs$ is empty or not. 
Clearly, in Example~\ref{ex:generalvsplit}, box $B_{21}$ cannot be split with respect to the first component, but it can be split with respect to the second and the third component like a `regular' non-redundant box. }

In order to distinguish the redundant boxes that appear in the case that 
a point equals a previously generated point in one component from the actual redundant boxes, we call the former boxes \emph{quasi non-redundant} boxes. 
Evidently, Lemma~\ref{lem:neighb} does not hold if \qnr boxes are part of the decomposition. 
However, the neighborhood structure which was obtained under Assumption~\ref{assump:znew2}~(1) can be preserved in the presence of \qnr boxes if we 
use a recursive update of the neighbors that is analogous to the non-redundant case, i.e., that uses the same sorting of the boxes. 
Then, $\Bsi(B)$ can be set as derived in Corollary~\ref{cor:proof}. This in turn means that the $v$-split as well as Algorithm~\ref{table:algoVsplit} do not need to be changed, but can be applied also when Asssumption~\ref{assump:znew2}~(1) is removed. Thus, Theorem~\ref{cor:runtime} which shows that the number of subproblems is bounded by $3 |\NS|-2$ holds independently of Assumption~\ref{assump:znew2}~(1). 
Finally, we revisit and extend the previous example in order to illustrate how Algorithm~\ref{table:algoVsplit} is applied in the presence of \qnr boxes.

\begin{example}\label{exo:quasinonred}
As in Example~\ref{ex:generalvsplit}, let 
$\mathcal{B}_{1}:=\{ B_{0} \}$ with 
$$
B_0:=\{z \in Z \, : \, 0 \leq z_{i} \leq 5 \; \forall \, i=1,2,3 \}
$$
be a given initial decomposition, 
and let $z^{1}=(3,1,4)^{\top}$ and $z^{2}=(3,2,1)^{\top}$ be two nondominated points.  
Then, the decomposition of the search region (including the \qnr box $B_{21}$) at the beginning of the third iteration is
$\mathcal{B}_{3}:=\{B_{11}, B_{12}, B_{21}, B_{22}, B_{23}\}$ with
\begin{align*}
u(B_{11})&=(3,5,5)^{\top}, \quad v(B_{11})=(0,1,4)^{\top},\\
u(B_{12})&=(5,1,5)^{\top}, \quad v(B_{12})=(3,0,4)^{\top},\\
u(B_{21})&=(3,5,4)^{\top}, \quad v(B_{21})=(3,2,1)^{\top},\\
u(B_{22})&=(5,2,4)^{\top}, \quad v(B_{22})=(3,1,1)^{\top},\\
u(B_{23})&=(5,5,1)^{\top}, \quad v(B_{23})=(3,2,0)^{\top}.
\end{align*}
The corresponding individual subsets are depicted in Figure~\ref{fig:quasinonred1}~(a). 
Note that the empty individual subset of the \qnr box $B_{21}$ is illustrated as the two-dimensional face 
$$
\left\{ z \in B_{0} :  v(B_{21}) \leqq z \leqq u(B_{21}) \right\}.
$$

Let now as third nondominated point 
$z^{3}=(2,2,2)^{\top}$ be given. 
As $z^3$ is contained in $B_{11}$ and $B_{21}$, we consider 
$v(B_{11})=(0,1,4)^{\top}$ and 
$v(B_{21})=(3,2,1)^{\top}$ for the $v$-split.  
Comparing $z^3$ with these two vectors reveals that $B_{11}$ is split with respect to the first and the second component, and $B_{21}$ is split with respect to the second and the third component, which yields
$$
u(B_{31})=(2,5,5)^{\top},
u(B_{32})=(3,2,5)^{\top},
u(B_{32}')=(3,2,4)^{\top},
u(B_{33})=(3,5,2)^{\top}.
$$
Hence, $\mathcal{B}_{4}:=\{B_{12}, B_{22}, B_{23}, B_{31}, B_{32}, B_{32}' , B_{33}\}$.
As $B_{31}$ is the only box obtained by a split with respect to the first component, 
we obtain $v(B_{31})=(0,2,2)^{\top}$. Analogously, $v(B_{33})=(2,2,1)^{\top}$.
For the update of $v(B_{32})$ and $v(B_{32})'$, the upper bound vectors 
$u(B_{32})$ and 
$u(B_{32}')$
are ordered increasingly with respect to one component $j \neq 2$ and decreasingly with respect to the remaining component $k \neq j, k \neq 2$, e.g., $j=1$ and $k=3$. 
As $u_{1}(B_{32}) = u_{1}(B_{32}')$ and $u_{3}(B_{32}) > u_{3}(B_{32}')$, we can order the boxes strictly decreasingly with respect to $k=3$. Therefore,
$$
v(B_{32})=(2,1,4)^{\top} \quad \text{and} \quad
v(B_{32}')=(3,2,2)^{\top}
$$
is obtained. 
Note that $v_{1}(B_{32}')=u_{1}(B_{32}')$ and $v_{2}(B_{32}')=u_{2}(B_{32}')$.
The \qnr box $B_{32}'$ is completely contained in $B_{22}$ and $B_{32}$, 
as $(3,2,4)^{\top} \leqq (5,2,4)^{\top}$ and $(3,2,4)^{\top} \leqq (3,2,5)^{\top}$, respectively.  

The individual subsets of all $B \in \mathcal{B}_{4}$ are depicted in Figure~\ref{fig:quasinonred1}~(b).
As the individual subset $V(B_{32}')$ is empty, we depict the set
\[
\{ z \in B_{0} :  v(B_{32}') \leqq z \leqq u(B_{32}') \} = 
\{ z \in B_{0} :  (3,2,2)^{\top} \leqq z \leqq (3,2,4)^{\top} \}
\]
instead, which describes a one-dimensional face. 
It is represented as a black line in Figure~\ref{fig:quasinonred1}~(b).
 \end{example}

\begin{figure}
\subfigure[]{
\includegraphics[width=.49\textwidth]{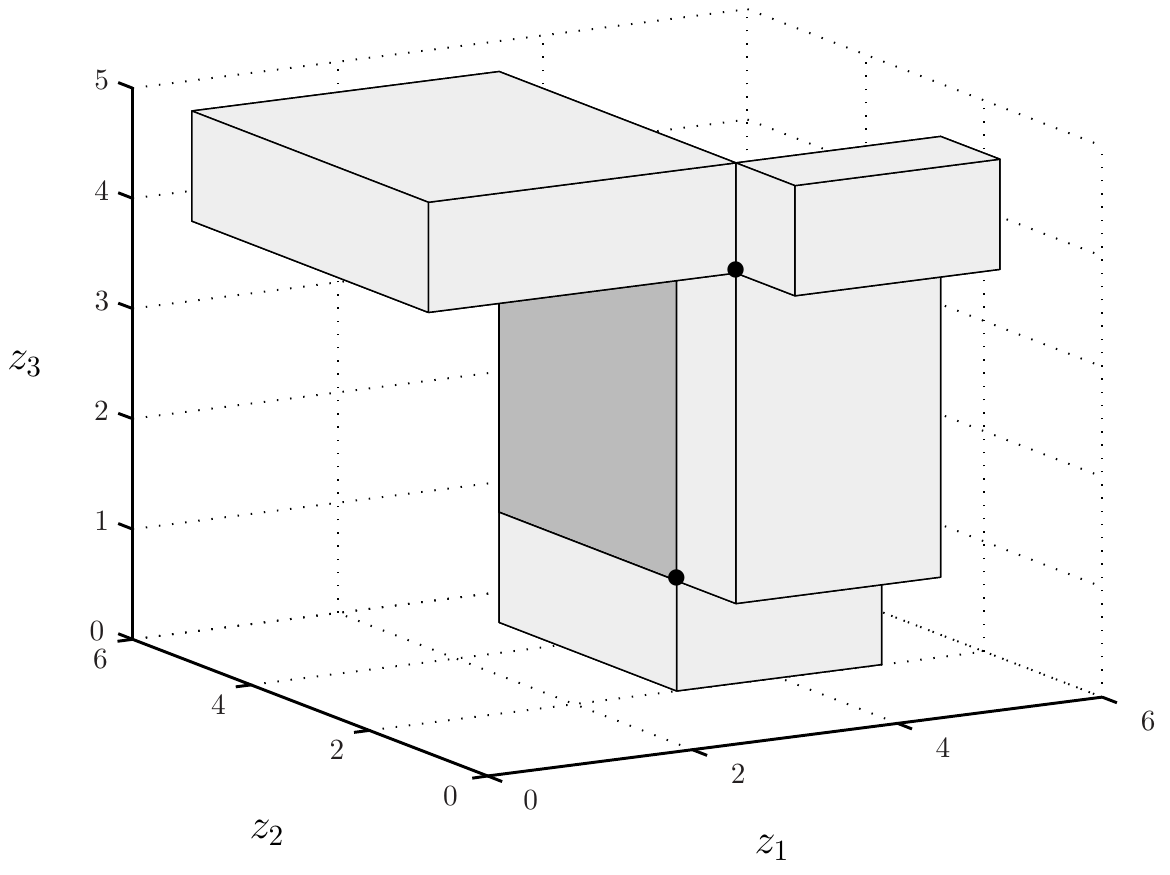}
}
\subfigure[]{
\includegraphics[width=.49\textwidth]{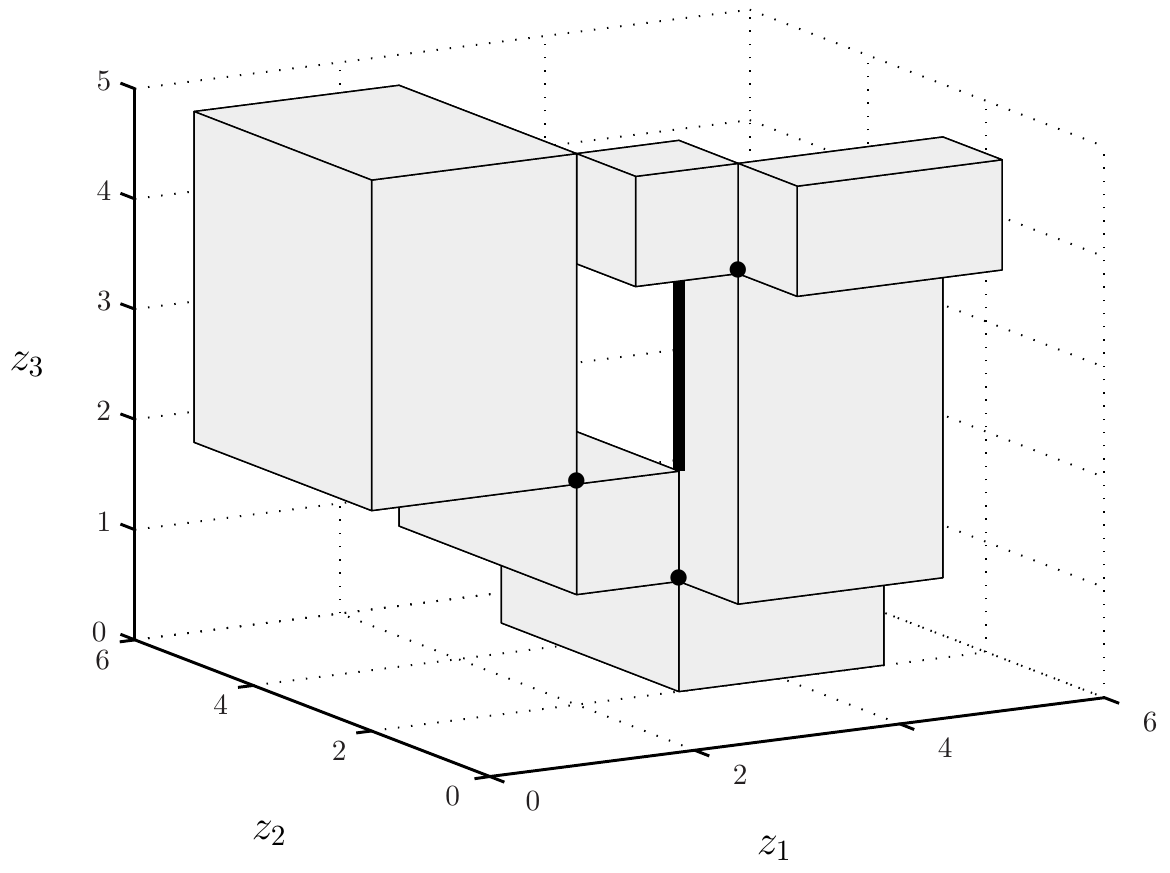}
}
\caption{Illustration of the sets $V(B)$ in Example~\ref{exo:quasinonred};
In (a) the individual subset of the occurring \qnr box (which is actually empty) is represented as a slightly darker, two-dimensional face, in (b) as a black one-dimensional face, i.e., a line.
 }
\label{fig:quasinonred1}
\end{figure}

Figure~\ref{fig:largeexo} shows an example with $68$ nondominated points for which Assumption~\ref{assump:znew2}~(1) does not hold. 
After having determined the initial search box, $3 |\NS|-2=202$ subproblems are solved until the termination criterion of Algorithm~\ref{table:algoVsplit} is reached, i.e., the upper bound derived in Theorem~\ref{cor:runtime} is sharp and holds also when \qnr boxes occur.  

\begin{figure}
\centering
\includegraphics[width=.55\textwidth]{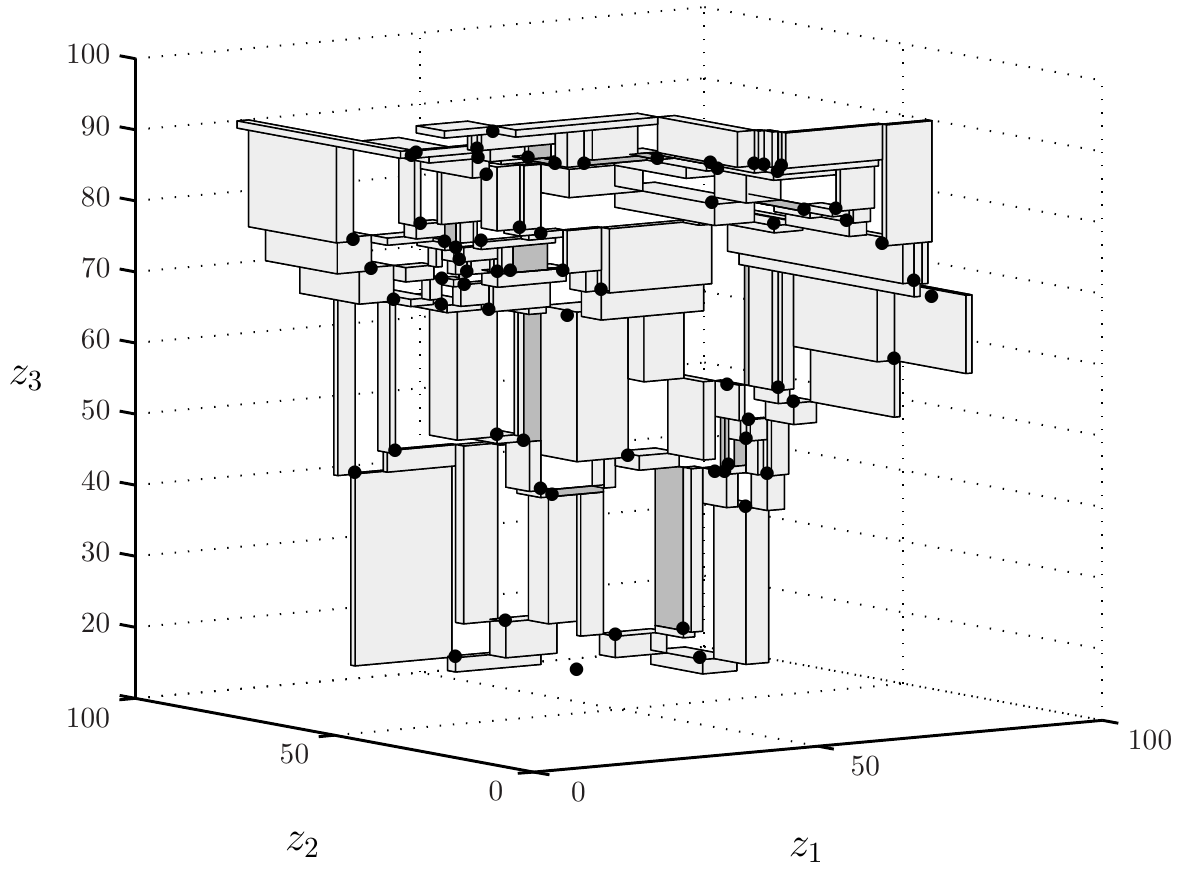} 
\caption{Individual subsets of all boxes of the final decomposition for an example with $68$ nondominated points}
\label{fig:largeexo}
\end{figure}
 
%
%
\section{Using the $\varepsilon$-constraint method as scalarization}
 \label{sec:econstr}

Algorithm~\ref{table:algoVsplit} presented in Section~\ref{sec:vsplit} is formulated independently of a specific scalarization. 
In every iteration, only points that are dominated by the current nondominated point are eliminated.  
In this section we show that we can reduce the search region and, thereby, the number of subproblems further if we use the \ec{} method. 
The specific properties of the \ecm are also used, e.g., in 
\citet{laumanns06}, \citet{lokman13}, \citet{oezlen13} and \citet{kirlik14}.
However, as to the best of our knowledge we present the first algorithm for tricriteria problems whose number of subproblems is proven to depend linearly on the number of nondominated points, the bound which is derived in this section is new as well. 

The reduction of the search region stems from the following 
property of the $\ve$-constraint method, which holds for any number of criteria. 
First, recall that for every $z^* \in \NS$, 
by definition of nondominance, we can exclude the two sets
$$
S_{1} (z^*):= \{z \in B \, : \, z \leqq z^* \}  \quad \textnormal{and} \quad
S_{2} (z^*):= \{z \in B \, : \, z \geqq z^* \}
$$
from 
every box $B$ that contains $z^*$.
If the point $z^*$ has been obtained as an optimal solution of an $\ve$-constraint problem of the form
\begin{align} \label{prob:lexec}
\min \quad & z_{1} \notag \\
\textrm{s.t.} \quad & z_{i} < u_{i}(\bar{B}) \quad \forall \, i =2,\dots,m,
\end{align}
where $\bB$ is a box of the current decomposition, then, additionally, the set
$$
S_{1}'(z^*):= \{z \in \bB \, : \, z_{1} < z_{1}^* \} 
= \{z \in \R^m \, : \, z_{1} < z_{1}^*, \; z_{i} < u_{i}(\bB)  \; \forall \, i=2,\dots,m \}
$$
cannot contain any further points, as this would contradict the optimality of $z^*$ in \eqref{prob:lexec}. 
Note that $S'_{1}(z^*)$ depends on $\bB$ as well as on the component~$i$ with respect to which the \ec{} problem is minimized. 
We choose $i=1$ without loss of generality. 
Also note that an optimal solution of \eqref{prob:lexec} is only weakly efficient, in general. Hence,
in order to guarantee that $z^*$ is nondominated,
a lexicographic \citep[see, e.g.,][]{laumanns06}, a two-stage \citep[see, e.g.,][]{kirlik14} or an augmented \citep[see, e.g.,][]{lokman13} \ec{} method should be employed. 
When solving~\eqref{prob:lexec}, two cases might occur. 
If $z^*_{1} \geq u_{1}(\bB)$, then $\bB \subseteq S_{1}'(z^*)=\emptyset$ holds. 
This case corresponds to the situation in which the current subproblem is infeasible, see Line~\ref{table:algoVsplitempty} in Algorithm~\ref{table:algoVsplit}. Box $\bB$ is removed from the decomposition and a new iteration starts. 
Otherwise, i.e., if $z^*_{1}<u_{1}(\bB)$, $z^* \in \bB$ holds. 
In Figure~\ref{fig:exclusion}, an example of the sets $S_{1}$, $S_{1}'$ and $S_{2}$ for $z^* \in \bB$ is depicted. 

\begin{figure}
\centering
\includegraphics[width=.9\textwidth]{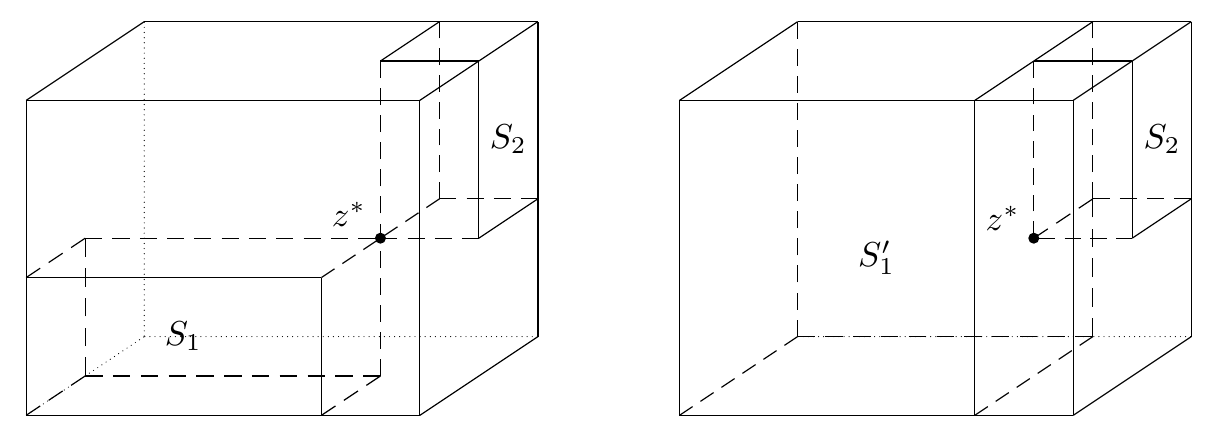}
\caption{Reduction of the search region for $m=3$: Solely based on nondominance of $z^*$ (left) and when taking into account that $z^*$ is obtained as optimal point of a corresponding \ecm (right)}\label{fig:exclusion}
\end{figure}

We consider now the implications of 
this additional reduction of the search region in combination with the $v$-split algorithm. 
Let box $\bB \in \mbs$ be the currently selected box, and let $z^s$ denote the nondominated point obtained in iteration~$s$, $z^s \in \bB$.
Then, the set $S_{1}'(z^s):=\{z \in \bB \, : \, z_{1} < z_{1}^s \}$ corresponds to the box obtained by a split of $\bB$ with respect to the first component. 
Since $S'_{1}(z^s)$ is empty, $\bB$ does not need to be split with respect to the first component. 
A split with respect to all other components $i \in \{2,3\}$ is performed according to the $v$-split criterion, i.e., if and only if $z^s_{i} \geq v_{i}(\bB)$ holds for $i \in \{ 2,3\}$.
For all other boxes $B \in \mbsb \backslash \{\bB\}$ the usual $v$-split criterion is employed with respect to all components. 
In particular, a box $B \in \mbsb \backslash \{\bB\}$ must be split with respect to the first component whenever $z^s_{1}\geq v_{1}(B)$ holds.

In order to benefit from the fact that the set $S_{1}'(z^s)$ can be excluded additionally from the search region, we must guarantee that the box resulting from a split of $\bB$ with respect to the first component  would have been part of the decomposition.
According to the definition of the $v$-split, this is the case if $z^s_{1} \geq v_{1}(\bB)$ holds. 
A sufficient criterion to guarantee that $z^s_{1} \geq v_{1}(\bB)$ holds is to select a box $\bB$ that does not have a neighbor in $\mbs$ with respect to $i=1$, i.e., 
$\Bsi(\bB)=\emptyset$.
Equivalently, we might select a box $\bB$ which satisfies $v_{1}(\bB)=\min\{v_{1}(B) : B \in \mbs\}$.
This means that we replace
Line~\ref{algoV:select} in Algorithm~\ref{table:algoVsplit} by 
`choose $\bB \in \mbs$ such that $v_{1}(\bB)=\min\{v_{1}(B) : B \in \mbs\}$'.
If a box with minimal value~$v_{1}$ is selected according to this rule at the beginning of each iteration, then one box is saved in each iteration in which a new nondominated point, that does not equal the ideal point in the first component, is generated in the selected box. 
Therefore, we obtain $2|\NS|-1$ as new upper bound on the number of subproblems to be solved in the tricriteria case. 

%
%
\section{Numerical results} \label{sec:num}

For our tests we use five instances of a tricriteria multidimensional knapsack problem, i.e., a zero-one knapsack problem with three objectives and three constraints. The considered instances have already been employed for numerical experiments, e.g., in \citet{laumanns05} and \citet{oezlen13}, wherefore we regard them as a good benchmark.
The five instances correspond to five different numbers of (knapsack) items $n=10,20,30,40,50$. The respective cardinality of the nondominated set is 
$9,61,195, 389$ and $1048$, as reported in \citet{laumanns05} and \citet{oezlen13} and verified by our algorithms. 
Note that we generated and saved the nondominated set of every instance once. 
For all methods presented in the following 
we always compare the respective generated representation with the saved nondominated set in order to verify that the complete nondominated set is computed correctly.
The computational platform for our study is a compute server with 4x Intel Xeon E7540 CPUs (2.0 GHz) and 128 GB of memory.  
All algorithms are (re)implemented in MATLAB R2013a and call IBM ILOG CPLEX Optimization Studio Version 12.5 to solve the subproblems. 
We turned off the option of CPLEX to parallelize.

\subsection{Validation of the $v$-Split Algorithm}

We test Algorithm~\ref{table:algoVsplit} in combination with a weighted Tchebycheff method (WT) and an \ecm (EC). 
In particular, we are interested in the question whether the upper bounds on the number of subproblems $3|\NS|-2$ (WT) and $2|\NS|-1$ (EC), which were derived in Sections~\ref{sec:vsplit} and~\ref{sec:econstr}, can be validated numerically. 
Both scalarizations are tested in an augmented and a two-stage formulation, see \citet{steuer83} in case of the \wt scalarization.
Note that typically it makes a difference for computational time whether an augmented or a two-stage approach is used. 
This is caused by the fact that in the latter, two integer problems are solved in every subproblem in which the first stage yields a feasible solution. In contrast, when an augmented method is used, only one integer problem per subproblem is solved. 
The parameters of all scalarizations are set adaptively dependent on the upper bound vector of the box that is selected in the current iteration. 
For the two-stage variant, we solve 
\eqref{prob:lexec} in the first stage and
$$
\min \left\{ 
\sum_{i=1}^m z_{i} : z_{j} \leq z_{j}^*, j=1,\dots,m
\right\}
$$
in the second stage, where $z^*$ denotes the weakly nondominated point obtained in the first stage. 
For the augmented variant we 
exploit the given integrality in order to determine a suitable augmentation parameter. 
For a detailed description we refer to \citet{daechert14}.

In Table~\ref{chap4:tricrit:results1}, the results are reported. 
The CPU times (in seconds) are averaged over three independent runs. 
To facilitate the comparison of the number of subproblems with the values $3|\NS|-2$ and $2|\NS|-1$, respectively, we indicate these values in parentheses in the second column of Table~\ref{chap4:tricrit:results1}.

\begin{table}
\footnotesize
\centering
\begin{tabular}{cc|c|cc|cc}
\addspacerow 
\multirow{2}{*}{$n$}& \multirow{2}{*}{$|\NS|$} &  &
\multicolumn{2}{c|}{Algorithm~\ref{table:algoVsplit}  (WT)} & 
\multicolumn{2}{c}{Algorithm~\ref{table:algoVsplit}  (EC)} \\ \cline{4-7} 
\addspacerow & & & CPU & \#SP & CPU & \#SP   \\ 
\hline \hline 
\addspacerow  \multirow{2}{*}{10} & 9 & TS  & 10.03 & \multirow{2}{*}{25} &  7.97 & \multirow{2}{*}{17} \\
 \addspacerow  & $(25/17)^{\star}$ &  A  &  7.81 &  &  6.09 \\ \hline

 \addspacerow  \multirow{2}{*}{20} & 61 & TS  & 56.42 & \multirow{2}{*}{181}  & 43.29 & \multirow{2}{*}{121} \\ 
 \addspacerow  & $(181/121)^{\star}$ &  A  & 42.72 &  & 30.02  \\ \hline
 
\addspacerow  \multirow{2}{*}{30} & 195 & TS  & 213.31 & \multirow{2}{*}{583} & 163.15 & \multirow{2}{*}{389} \\ 
 \addspacerow  & $(583/389)^{\star}$ &  A  & 163.29 &  & 114.39 \\ \hline

 \addspacerow  \multirow{2}{*}{40} & 389 & TS  & 464.47 & \multirow{2}{*}{1165} & 361.74 & \multirow{2}{*}{777} \\ 
 \addspacerow  & $(1165/777)^{\star}$ &  A  & 361.01 &  & 257.64  \\ \hline

 \addspacerow  \multirow{2}{*}{50} & 1048 & TS  & 1552.56 & \multirow{2}{*}{3142}  & 1369.89 & \multirow{2}{*}{2095} \\ 
 \addspacerow  & $(3142/2095)^{\star}$ &  A  & 1174.90 & &  1012.15 \\ \hline
  \hline 
\end{tabular}
\caption{Average CPU times (in seconds) and number of subproblems solved 
by Algorithm~\ref{table:algoVsplit} in combination with a weighted Tchebycheff method (WT) and an \ecm (EC). 
Each scalarization is evaluated in a two-stage (TS) and an augmented (A) formulation, respectively.   
In the second column, additionally to $|\NS|$, the theoretical upper bounds 
$3|\NS|-2$ (WT) and $2|\NS|-1$ (EC) are given in parentheses $()^{\star}$ for better comparison.  
}
\label{chap4:tricrit:results1}
\end{table}

Consider first the number of subproblems solved.
From Table~\ref{chap4:tricrit:results1} we see that Algorithm~\ref{table:algoVsplit}~(WT) requires exactly 
$3 |\NS|-2$ and Algorithm~\ref{table:algoVsplit}~(EC) exactly 
$2 |\NS|-1$ subproblems for all problem sizes and for both formulations, i.e., for a two-stage (TS) and an augmented (A) formulation.
Hence, the predicted upper bound on the number of subproblems is met precisely. 

Regarding computational times in Table~\ref{chap4:tricrit:results1} we observe that all variants using (EC) are considerably faster than the variants using (WT) as the former solve about one third less subproblems compared to the latter. 
However, the savings with respect to computational time are not proportional to the savings with respect to the number of subproblems, in general.
Recall from Section~\ref{sec:econstr} that in order to achieve a saving with respect to the number of subproblems when using the \ec{} method, the box at the beginning of each iteration cannot be selected arbitrarily, but a box with minimal value~$v_{1}$ must be identified. 
This causes an additional computational effort. 

\subsection{Comparison of Three Recent Algorithms to the New Algorithm}

We additionally compare Algorithm~\ref{table:algoVsplit} to three recent algorithms for generating complete representations for discrete multicriteria optimization problems with finite nondominated set.
These comprise the second algorithm stated in \citet{lokman13}, the approach of \citet{kirlik14} and the method of \citet{oezlen13}. 
All three methods employ an \ec{} scalarization, however, each in a different variant: 
\citet{lokman13} use an augmented, 
\citet{kirlik14} a two-stage and
\citet{oezlen13} a lexicographic \ec{} method.
In order to make the comparison in our numerical study as fair as possible, we implement and test all algorithms with both, a two-stage and an augmented formulation,
where the augmentation parameter is set adaptively according to the formulas presented in \citet{daechert14}.
Note, however, that thereby we 
extend and/or modify the original algorithms of \citet{lokman13}, \citet{kirlik14} and \citet{oezlen13}. 
Dependent on the formulation and particular parameters (as the augmentation parameter) used, the order in which the points are generated might change. However, this does not affect the general functionality of the respective algorithms.   

In the literature, 
further methods to compute complete representations of discrete multicriteria optimization problems are presented, see, e.g., \citet{tenfelde03}, \citet{sylva04}, \citet{laumanns05}, \citet{laumanns06} and \citet{oezlen09}. 
However, all these approaches have been reported to be outperformed by at least one of the three methods that we incorporate into this numerical study. 
 
We reimplement the algorithms of Lokman and K\"oksalan (LK), Kirlik and Say{\i}n (KS) and \"Ozlen, Burton and MacRae (OBM) with the following slight modifications. 
The algorithm of \citet{lokman13} is originally formulated for problems in maximization format. 
For the sake of simplicity, we implement it for minimization problems. 
Moreover, as recommended in \citet{lokman13}, we keep the list of current nondominated points sorted, as, thereby, better computational times are obtained.
In (LK) and (OBM), the right-hand side vectors of previously solved subproblems are saved as well as the corresponding results, i.e., a (nondominated) point or a value indicating infeasibility. 
Before solving a subproblem, the list of bounds is scanned to find a so-called relaxation. 
 If a relaxed problem exists and it is either infeasible or the saved point is feasible for the current subproblem, then the current subproblem 
 does not need to be solved since the solution of the relaxation is also valid for the considered subproblem.
 In this case, the bounds of the current subproblem should not be saved, as they do not contribute new information and, clearly, the shorter the list of bounds is, the better computational times can be expected. 
In the implementation of (KS) we change a detail with respect to the pseudocode given in \citet{kirlik14}. When a new nondominated point is generated, all cells of the decomposition are checked twice in \citet{kirlik14}: first, to identify the cells to be split, secondly, to remove cells that can not contain further nondominated points. 
We combine both checks, which are performed within two independent procedures in the original pseudocode, into one by removing cells that can not contain further nondominated points immediately after or during the split.   
In our implementation, this slight modification led to a huge saving of computational time.

The CPU times and the number of subproblems solved by all methods and for all instances are given in Table~\ref{chap4:tricrit:results3}. 
Again, the given CPU times are averaged over three independent runs. 

\begin{table}
\footnotesize
\centering
\begin{tabular}{cc|c|cc|cc|cc|cc}
\addspacerow
\multirow{2}{*}{$n$}& \multirow{2}{*}{$|\NS|$}& & 
\multicolumn{2}{c|}{LK} & \multicolumn{2}{c|}{KS} & \multicolumn{2}{c|}{OBM} & \multicolumn{2}{c}{Algorithm~\ref{table:algoVsplit}  (EC)} \\
\cline{4-11}
\addspacerow
& && CPU & \#SP & CPU & \#SP & CPU & \#SP & CPU & \#SP \\  
\hline \hline

\addspacerow  \multirow{2}{*}{10} & \multirow{2}{*}{9} & TS  &  9.48 & \multirow{2}{*}{20} &  8.50 & \multirow{2}{*}{17} &  8.85 & \multirow{2}{*}{19} & 7.97 & \multirow{2}{*}{17} \\ 
 \addspacerow  &  &  A  &  6.67 & &  6.07 & &  6.46 & & 6.09 & \\ \hline

 \addspacerow  \multirow{2}{*}{20} & \multirow{2}{*}{61} & TS  & 53.04 & {127} & 50.08 & \multirow{2}{*}{115} & 48.50 & \multirow{2}{*}{117} & {43.29} & \multirow{2}{*}{121} \\ 
\addspacerow  & &  A  & 31.76 & {128} & 30.26 & & 28.83 & & {30.02} & \\ 
\hline

 \addspacerow  \multirow{2}{*}{30} & \multirow{2}{*}{195} & TS  & 267.88 & {468} & 242.42 & {373} & 197.05 & {375} & {163.15} & \multirow{2}{*}{389} \\ 
 \addspacerow  &  &  A  & 159.12 & {464} & 155.89 & {372} & 110.33 & {374} & {114.39} \\ \hline

 \addspacerow  \multirow{2}{*}{40} & \multirow{2}{*}{389} & TS  & 657.58 & \multirow{2}{*}{852} & 701.95 & {739} & 430.84 & {741} & {361.74} & \multirow{2}{*}{777} \\ 
 \addspacerow  &  &  A  & 445.07 & & 516.15 & {738} & 246.68 & {740} & {257.64} & \\ \hline
 
 \addspacerow  \multirow{2}{*}{50} & \multirow{2}{*}{1048} & TS  & 4772.89 & {2193} & 4174.48 & {1913} & 1533.93 & {1915} & {1369.89} & \multirow{2}{*}{2095} \\ 
 \addspacerow  &  &  A  & 4129.47 & {2200} & 3603.67 & {1914} & 945.35 & {1916} & {1012.15} \\ \hline
 \hline
\end{tabular}
\caption{Average CPU times (in seconds) and number of subproblems solved by three state of the art algorithms and Algorithm~\ref{table:algoVsplit}~(EC). 
Each scalarization is evaluated in a two-stage (TS) and an augmented (A) formulation, respectively.   
}
\label{chap4:tricrit:results3}
\end{table}

Regarding the number of subproblems solved, we observe that (KS) generates a complete representation within the lowest number of subproblems among all compared methods in all instances.
Method~(LK) requires the largest number of subproblems in all instances. 
While methods (OBM), (KS) and Algorithm~\ref{table:algoVsplit}~(EC), except (OBM) for $n=10$, solve at most $2 |\NS|-1$ subproblems, 
(LK) exceeds this bound in all instances. 
These results go in line with the results of \citet{lokman13}, who state that they solved on average $2.08$ subproblems per nondominated point in their numerical study for a classic (one-dimensional) tricriteria knapsack problem.
Our results also coincide 
with the results of \citet{kirlik14}, who state that they required on average $1.97$ and at most $1.99$ subproblems per nondominated point with their algorithm when it was applied to a classic tricriteria knapsack problem.
In our study, (KS) even performs better. In the worst case ($n=30$) less than $1.92$ subproblems per nondominated point are solved.  
The results of (OBM) can be compared directly with the results reported in \citet{oezlen13}, as they solve the same problem with the same instances. 
In their numerical study, $46, 333, 1204, 2357$ and $6001$ subproblems are solved for $n=10,\dots,50$, respectively. 
Interestingly, we obtain a considerably smaller number of subproblems with our reimplementation in all instances. 
A possible reason for this mismatch might be 
the scalarization used. 
While we apply (OBM) in combination with a two-stage and an augmented scalarization, a lexicographic \ec{} scalarization is used in \citet{oezlen13}, 
which might lead to a higher number of subproblems.

Considering CPU times, we obtain a slightly different picture. 
For the small instance $n=10$, the CPU times of all methods are quite close. 
For all other problem sizes, the best CPU times are clearly obtained by Algorithm~\ref{table:algoVsplit}~(EC) and (OBM). 
When the augmented formulation is used, (OBM) consumes less CPU time than Algorithm~\ref{table:algoVsplit}~(EC).
When the two-stage formulation is used, Algorithm~\ref{table:algoVsplit}~(EC) outperforms (OBM) for all problem sizes.

Both other methods, i.e., (LK) and (KS) require considerably more CPU time than (OBM) and Algorithm~\ref{table:algoVsplit}~(EC) for $n=20,30,40,50$. 
Besides $n=40$, (LK) performs worst. 
As (LK) solves more subproblems than all other methods, this result is not surprising. 
In contrast, the rather bad performance of (KS) is not expected with regard to the fact that (KS) solves the lowest number of subproblems in basically all instances.
The reason lies in the huge number of cells, which are maintained in (KS) and which are scanned several times during each iteration. 
This computational effort is reflected in the CPU times. 

We summarize that our new algorithm based on the $v$-split generates complete representations within the predicted number of subproblems. Moreover, it competes with state of the art algorithms.

\section{Conclusion}
In this paper, we positively answer the question whether there exists an algorithm which generates the entire nondominated set of a problem with more than two objectives by solving a number of subproblems which depends linearly on the number of nondominated points. We construct an algorithm which requires a linear number of subproblems for tricriteria problems. 
This is achieved by avoiding the generation of redundant boxes and by using neighborhood properties between the boxes. 
Further research should analyze whether and how the concept of individual subsets can be transferred to problems with more than three criteria. 
Moreover, the presented algorithm can be improved further by using the neighborhood properties also for identifying all boxes containing a current point. Thereby, no exhaustive search is needed in each iteration and the boxes can be updated more efficiently.
Finally, with slight modifications, the new algorithm can also be used
if only a representative subset of the nondominated set shall be generated.


\section{Acknowledgement}
The final publication is available at Springer via \href{http://dx.doi.org/10.1007/s10898-014-0205-z}{http://dx.doi.org/10.1007/s10898-014-0205-z}.

\bibliographystyle{myabbrvnat} 
\bibliography{dissbibliography}

\end{document}